\documentclass[psamsfonts, reqno]{amsart}
\usepackage{amssymb,amsfonts}
\usepackage{amsmath}
\usepackage[all,arc]{xy}
\usepackage{enumerate}
\usepackage{mathrsfs}
\usepackage{pstricks}
\usepackage{pst-node}
\usepackage{float}

\newtheorem{thm}{Theorem}[section]

\newtheorem{prop}[thm]{Proposition}

\theoremstyle{definition}
\newtheorem{defn}[thm]{Definition}

\newtheorem{exmp}[thm]{Example}

\theoremstyle{remark}
\newtheorem{rem}[thm]{Remark}

\newcommand{\sA}{{\mathcal A}}
\newcommand{\sC}{{\mathcal C}}
\newcommand{\sE}{{\mathcal E}}
\newcommand{\sG}{{\mathcal G}}

\newcommand{\sP}{{\mathcal P}}

\newcommand{\sD}{{D}}
\newcommand{\QQ}{{\mathbb{Q}}}
\newcommand{\NN}{{\mathbb{N}}}
\newcommand{\RR}{{\mathbb{R}}}
\newcommand{\ZZ}{{\mathbb{Z}}}
\newcommand{\bw} {{\bf w}}

\numberwithin{equation}{section}

\newcommand{\stdnodesep}{3}
\newcommand{\doffset}{3pt}

\newcommand{\node}[3]{\rput{0}(#2){\ovalnode{#3#1}{\Large #1}}}

\newcommand{\aline}[3]{%
	\ncline[nodesepA=\stdnodesep,nodesepB=\stdnodesep]%
	{->}{#1}{#2}%
	\Aput{#3}%
}

\newcommand{\dline}[4]{%
	\ncarc[nodesepA=\stdnodesep,nodesepB=\stdnodesep,offset=\doffset]%
	{->}{#1}{#2}%
	\Aput{#3}%
	\ncarc[nodesepA=\stdnodesep,nodesepB=\stdnodesep,offset=\doffset]%
	{->}{#2}{#1}%
	\Aput{#4}%
}

\newcommand{\bcircle}[3]{%
	\nccircle[angleA=#2,nodesepA=\stdnodesep]{->}{#1}{20pt}%
	\Bput{#3}%
}

\title{$P$-Adic Path Set Fractals and Arithmetic}

\author{William  Abram and Jeffrey C. Lagarias}
\subjclass[2010]{Primary: 11K55, Secondary:  11S82, 28A80, 37B10}
\keywords{p-adic arithmetic, finite automata, graph-directed systems, Hausdoff dimension}

\date{November 10, 2013}
\thanks{Work of the second author was partially supported by NSF grant DMS-1101373.}
\begin{document}

\begin{abstract}

This paper considers a class 
of closed subsets of the $p$-adic integers $\ZZ_p$ obtained by 
graph-directed constructions analogous to that of Mauldin and Williams
over the real numbers. These sets  are characterized as the
collection of those $p$-adic integers
whose points have $p$-adic expansions describable  by  paths 
 in the graph of a finite automaton issuing from a distinguished initial vertex.
 This paper shows  that this class of sets is  closed under the arithmetic
operations of 
 addition and multiplication by $p$-integral rational numbers.
 In addition the Minkowski  sum (under $p$-adic addition) of two sets in this class is shown
 to be a set in this class. These results represent purely $p$-adic phenomena
 in that analogous closure properties do not hold over the real numbers. 
 We also show the existence of computable  formulas  for  the Hausdorff dimensions of  such sets.

\end{abstract}

\maketitle


%
%
%
\section{Introduction}

This paper studies  a distinguished collection $\sC(\ZZ_p)$ of 
 closed subsets   of the $p$-adic integers $\ZZ_p$,
 whose  members $Y$ are specified as
  sets of $p$-adic integers whose 
$p$-adic expansions are  given by 
infinite labeled paths starting from a fixed initial state of a 
a  finite automaton, with edge  labels specifying
 $p$-adic digits.  We term such sets
{\em $p$-adic path set fractals,} because they generally
have non-integer Hausdorff dimension, and because they
 may be constructed geometrically in a fashion 
analogous to that of the (real-valued) geometric
graph-directed fractals of Mauldin and Williams \cite{MW86}, \cite{MW88},
as explained in Section \ref{sec2}.

The  set of edge-labeled infinite  paths in the graph of an automaton
which start  from a fixed state can be studied  abstractly 
 in terms of one-sided symbolic dynamics,
as we consider elsewhere (\cite{AL12a}).
 Each such set
 defines a subset $X_{\sG}(v)$ of the symbol space  $ \sA^{\NN}$,
 where $\sA$ is a finite symbol alphabet, 
 which is specified by a {\em presentation} $(\sG, v)$, in which $\sG$ is a labeled directed graph with edges
 labeled by elements of $\sA$, and $v$ is a marked initial vertex of $\sG$.  We call
 $X_{\sG}(v)$ an (abstract) {\em path set}. Path sets are closed subsets of the compact set $\sA^{\NN}$
 endowed with the product topology,
  but they are
 generally not invariant under the (one-sided) shift $\sigma: \sA^{\NN} \to \sA^{\NN}$
 given by $\sigma(\alpha_0, \alpha_1, \alpha_2, \cdots) = (\alpha_1, \alpha_2, \alpha_3 , \cdots)$.
 The collection of path sets is closed under set union and set intersection, but is not closed under complementation
 inside the symbol space $\sA^{\NN}$.
 
 A {\em $p$-adic path set fractal} $Y$ is  the image of an
 abstract  path set  embedded as  a geometric object inside 
 a $p$-adic space $\ZZ_p$, 
 using the  symbol sequence to obtain the $p$-adic digit labeling. 
 More precisely,  $Y = f_p( X_{\sG}(v))\subset \ZZ_p$ where
 $f_p : \sA^{\NN} \to \ZZ_p$ is a (continuous) map that sends a
 symbol sequence to a set of $p$-adic digits, using a {\em digit assignment map} $\bar{f}_p: \sA \to \{0, 1, ..., p-1\}$. 
 The digit assignment map $\bar{f}_p$ need not be one-to-one, consequently  a given  abstract path set $X_{\sG}(v)$ has 
  embeddings into  each $\ZZ_{p}$ for every prime $p$, moreover it typically has different embeddings into a fixed $\ZZ_p$,
  giving rise to different $p$-adic path set fractals. 
  In the reverse direction any given $p$-adic path set fractal can 
   be obtained as the image of different abstract path sets using different digit assignment maps.
  However any $p$-adic path set fractal  $Y$ can always be obtained by a one-to-one embedding from a suitably
 chosen  path set on the fixed alphabet $\sA=\{0, 1, ..., p-1\}$, 
  see Proposition \ref{pr30} below. We call the data $(\bar{f}_p, \sG, v)$ a {\em presentation} of the $p$-adic path set
 fractal, and write  $Y:=(\bar{f}_p, \sG, v)$.
  
  The initial part of this paper gives a formula for the  Hausdorff dimension of
  such a set $Y$  
   in terms of the spectral radius of the adjacency matrix of an underlying automaton
  of a suitable presentation of $Y$, described in Section \ref{sec3}.
  The Hausdorff dimension of  $Y= f_p(X_{\sG}(v))$ 
  depends  on the underlying
  path set $X_{\sG}(v)$, the value of $p$,  
 and on the  digit assignment map $f_p$.
  We obtain the formula by relating 
  the $p$-adic constructions  of this paper to the real number constructions of
 Mauldin and William, which permit  carrying
 over their  formulas  for  Hausdorff dimension of (real) graph-directed fractals
 to the $p$-adic case.

The main object of this paper is to show that 
the collection of $p$-adic path set fractals $\sC(\ZZ_p)$ is closed  under
the following operations using $p$-adic arithmetic: 
\begin{enumerate}
\item
 $p$-adic addition of a  rational number $r \in \QQ \cap \ZZ_p$; 
 such $r$ are called {\em $p$-integral.}
 \item
$p$-adic multiplication by a $p$-integral rational number $r$;
\item
set-valued addition (Minkowski sum) of two $p$-adic path sets, using $p$-adic addition.
\end{enumerate}
These closure results represent purely $p$-adic phenomena in the sense that analogous 
closure results for applying real arithmetic operations\footnote{That is, for   addition to and multiplication  of a set by rational numbers, 
or for the Minkowski sum of two sets.}
fail to hold for 
Mauldin-Williams 
graph-directed fractals over the real numbers.
We show that the  finite automata describing  the new sets  
given by  these operations are effectively computable;  
these new automata depend  on the input automata and on the value of  $p$
in a complicated way having a number-theoretic flavor.

%
%
%
\subsection{Results}\label{sec11}

As a preliminary to  $p$-adic results,    
 in  Section \ref{sec2}  we review the Mauldin-Williams construction
of graph-directed fractals over the real numbers.
We then  formulate an alternate definition of $p$-adic path set fractals, defining 
them geometrically as sets given  by a solution of a set-valued functional
equation using $p$-adic contracting maps 
(Theorem \ref{pr221} and Definition \ref{de26}). 
That is, they are characterized as a  
set-valued fixed point of a $p$-adic graph-directed fractal construction.
Then   we  establish  the equivalence of this geometric
definition  to  the symbolic dynamics 
definition  given above,  in terms of 
$p$-adic expansions describable by a finite automaton,
i.e.  the image  of a path set in $\ZZ_p$ under a digit assignment map. (Theorem \ref{self-sim}). 

To state  results, we need some additional  terminology on presentations.
\begin{enumerate}
\item
A presentation $Y :=(\bar{f}_p, \sG, v)$ of a $p$-adic path set fractal
is   {\em injective} if the digit assignment map $\bar{f}_p$ is one to one.
\item
A  path set presentation $(\sG, v)$ is  {\em right-resolving} if the  directed graph underlying $\sG$
has the property that  at each vertex of the graph, all exiting edges have different labels. 
\item
A path set presentation  $(\sG, v)$  is
{\em reachable} if every vertex in $\sG$ can be reached by a directed path from $v$.
\end{enumerate}
All $p$-adic path set fractals $Y$ have presentations that are injective,   right-resolving and reachable, 
see Proposition \ref{pr30}.
We call any such presentation a {\em standard presentation}. In a standard presentation one
may always choose to relabel the symbol alphabet  $\sA = \{0,1, 2,..., p-1\}$, 
and choose the digit assignment map to be the identity map.

Our first result concerns the Hausdorff dimension $d_{H}(Y)$ of a $p$-adic path set fractal
$Y = f_p(X_{\sG}(v))$. We 
show that $d_{H}(Y)$ is 
directly computable from a suitable
presentation of $Y$ as a $p$-adic path set fractal,
and is of an expected form.

\begin{thm}\label{th12} {\em (Hausdorff dimension)}
Let $Y$ belong to $\sC(\ZZ_p)$ and let $Y:=(\bar{f}_p, \sG, v_0)$
be any standard presentation.
  Then its Hausdorff dimension $ d_{H}(Y)$
is given by
\[
d_{H}(Y) = \log_p \sigma(A({\sG})) = \frac{ \log \sigma(A(\sG))}{\log p},
\]
where
$\sigma(A)$ denotes the spectral radius of the adjacency matrix
$ A =A(\sG) := [ a_{i,j}],$
 of $\sG$, in which $a_{i,j}$ counts the number of directed edges
from  vertex $i$ to vertex $j$ of the underlying directed graph  of $\sG$.
\end{thm}

This result is proved in Section \ref{sec3}, where it is deduced from a result
 relating these sets to real number graph-directed
fractals (Theorem \ref{th31a}), where a similar Hausdorff dimension
formula has long been known. We use the  fact  that Hausdorff dimension
is preserved under the map taking a $p$-adic expansion of a $p$-adic integer to the base $p$ radix expansion
of a real number. We show  
that   the image of  the $p$-adic  objects studied here 
are construction sub-objects of particular
 real number constructions of Mauldin and Williams.

The main results of this paper concern   $p$-adic arithmetic operations
applied to $p$-adic path set fractals.
We begin with  addition of $p$-adic rationals. 

\begin{thm}\label{th13} {\em (Closure under rational addition)}
Let $Y$ belong to  $\sC(\ZZ_p)$.   Then
for any $p$-integral rational number $r \in \QQ \cap \ZZ_p$, the additively shifted set  
\[
Y+ r := \{ y + r:  y \in Y\}
\]
has $Y+r \in \sC(\ZZ_p)$. 
\end{thm}

The proof of this result  is constructive and shows  
 that given a standard   presentation $Y := (\bar{i}_p, \sG, v)$
one can directly compute from it a standard  presentation  $Y+r := (\bar{i}_{p}, \sG', v')$.
The new presentation depends on both  $r \in \QQ$ and  the
value of $p$.

Theorem \ref{th13}  also follows as  a  special case of the following result.

\begin{thm}\label{th14} {\em (Closure under Minkowski sum)}
Let $Y_1, Y_2 \in \sC(\ZZ_p)$
be two $p$-adic path set fractals in $\ZZ_p$. Then
their Minkowski sum-set 
\[
Y_1 + Y_2 := \{ y_1 + y_2:  y_1 \in Y_1, \, y_2 \in Y_2\}
\]
has $Y_1 + Y_2 \in \sC(\ZZ_p).$
\end{thm}

The proof of this result is constructive
and shows that given standard presentations for
$Y_1:=(\bar{i}_p, \sG_1, v_1) $ and $Y_2= (\bar{i}_p, \sG_2, v_2) $ one can  directly construct a 
(not necessarily standard) presentation
$Y_1+ Y_2= (\bar{i}_{p}, \sG_3, v_3)$.
In the given construction the  underlying path set presentation
$(\sG_3, v_3)$ produced  is not necessarily right-resolving.
However  there exist  standard algorithms to convert any given 
path set presentation to one that is right-resolving, see \cite[Theorem 3.2]{AL12a}. 

The statement of Theorem ~\ref{th13} can be recovered 
from Theorem \ref{th14} as the special case that 
the set $Y_2$
is a one-element set, using the easy observation that 
the only path sets in $\ZZ_p$ consisting
of a single element are  those where this element is
a rational $r \in \QQ \cap \ZZ_p$ (Theorem ~\ref{th15}). 
However  the presentation  obtained by the construction
of Theorem \ref{th14} is not necessarily
right-resolving, while the construction given  in the proof of Theorem \ref{th13}
is right-resolving.

We next consider multiplication by $p$-adic rationals.

\begin{thm}\label{th16}{\em (Closure under rational multiplication)}
Let $Y$ belong to  $\sC(\ZZ_p)$. Then
for any rational number $r \in \QQ \cap \ZZ_p$, 
the dilated set 
\[
rY := \{ ry : y \in Y \}
\]
has $r Y \in \sC(\ZZ_p)$. 
\end{thm}

We prove this result in Section \ref{sec5}.
This proof is constructive in the same sense as Theorem \ref{th12};
given a standard presentation for $Y$ there is (in principle) an algorithm to find 
a standard presentation 
for $rY$.
Theorem \ref{th16} is obtained by  concatenation of constructions for
several special  cases, as follows.
\begin{enumerate}
\item
   $r=M$ is a 
positive integer with $\gcd(p, M)=1$.
A positive integer has an infinite  $p$-adic expansion
with a finite pre-period and a periodic part with all  digits $0$.
\item $r = \frac{1}{M}$ is  the inverse of a positive
integer $M$ with $\gcd(p, M) =1$;  
\item
$r= -1.$ Note that  $-1$ has a purely periodic nonterminating $p$-adic expansion of period $1$: 
$$
-1= \sum_{k=0}^{\infty} (p-1)p^k = (\cdots, p-1, p-1,p-1)_p.
$$

\item $r=p^k$, for $k \ge 1$. 
\end{enumerate}

Finally we note that, at the level of symbolic dynamics, 
the arithmetic operations are not compatible with the
one-sided shift operation. That is, if
  $Y$ is a $p$-adic path set which is 
 invariant under
the one-sided shift $\sigma_p: \ZZ_p \to \ZZ_p$ defined by
\[
\sigma_p ( \sum_{j=0}^{\infty} \alpha_j \,p^j ) := \sum_{j=0}^{\infty} \alpha_{j+1} \,p^j,
\]
 then in general  the sets $Y+r$, $rY$ 
(for $p$-integral rational $r$) will {\em not} be invariant under 
the one-sided shift $\sigma_p$.

%
%
%

\subsection{Extensions and generalizations}\label{sec12}

In Theorem \ref{th12}  we give a formulas for Hausdorff dimension of the resulting $p$-adic
path set fractals  in terms of  the spectral radius of a nonnegative integer matrix
specifying the graph of the underlying path set.
This formula might initially appear unnecessary in the context of 
the $p$-adic arithmetic operations studied, because 
given any set $X \subset \ZZ_p$ and 
 any nonzero $\alpha \in \ZZ_p$, rational or not, 
the sets $X$, $X + \alpha$ and $\alpha X$ all have the same Hausdorff dimension.
This fact follows since  both sets $X+ \alpha$ and $\alpha X$
are  images of $X$ under  bi-Lipschitz mappings in the $p$-adic metric. 
In this context,  the usefulness of  Theorem \ref{th12}  lies rather  in the opposite direction:
using it, the known  equality of Hausdorff dimensions  
yields  the equality of the spectral radii of two
quite different appearing nonnegative matrices, usually of different sizes.

It is of interest  that the constructions of this paper  will, when given a fixed set $X$
as input, by varying $r$ yield infinite classes of
nonnegative integer matrices having a fixed spectral radius. This spectral radius
always equals the largest real eigenvalue, which 
is a particular type of real algebraic integer called a  weak Perron number. Here
a {\em weak Perron number} is a  (positive real) $n$-th root of a  {\em Perron number} for some $n \ge 1$, cf.
Lind \cite{Lin84}, \cite{Lin87}. A {\em Perron number} is a real algebraic integer $\beta>1$,
all of whose algebraic conjugates $\beta^{\sigma}$ are smaller in modulus, i.e.
$|\beta^{\sigma}| < \beta.$ Such  classes of matrices may be  worth
further study in connection with number theoretic problems, see Section \ref{sec7}.

The constructions of this paper can be combined with other operations
which preserve the property of being a $p$-adic path set fractal but
which do change the Hausdorff dimension.
For example, path sets are  closed under set union
and set intersection  (\cite[Theorem 1.2]{AL12a}),
with the new path set presentations being effectively computable from the given ones.
Set union changes the Hausdorff dimension in a predictable way, with the new set 
having dimension equal to 
 the maximum of the two dimensions, however set 
 intersection changes  Hausdorff dimension in seemingly  unpredictable ways.
Given presentations of $p$-adic path set fractals  $Y_1$ and $Y_2$
one can, in principle,  compute the Hausdorff dimension of intersections of
additive and multiplicative translates of these sets, such as 
$Y_1 \bigcap (Y_2 + r)$ and $Y_1 \bigcap (r Y_2)$.
This  study was undertaken  to answer questions of this kind that arose
 in connection with a problem of Erd\H{o}s, see  Erd\H{o}s  \cite{Erd79}, and papers \cite{Lag09},  \cite{AL12c}
 of the authors.
 Computed examples in \cite{AL12c} illustrate that the Hausdorff dimensions of 
 sets  $Y \cap (Y+r)$ and $Y \cap rY$
 vary with $r$, and the dependence on $r$ of these Hausdorff dimensions
 appears to be extremely complicated, with  interesting structure.

The class $\sC(\ZZ_p)$ of $p$-adic path set  fractals are closed under another operation: {\em decimation}, i.e.
extracting a fixed  arithmetic progression of their $p$-adic digits. We set
$$
\psi_{j, m} (\alpha_0, \alpha_1, \alpha_2, \cdots) = ( \alpha_j, \alpha_{j+m}, \alpha_{j+2m}, \alpha_{j+ 3m}, \cdots).
$$
and the define the $(j, m)$-decimated  set
$$
\psi_{j,m}(Y) =\{ \psi_{j,m}(x): x \in Y\},
$$ 
The fact that $\psi_{j, m}(Y)$ set belongs to $\sC(\ZZ_p)$ if $Y$ does
follows at the path set level from  \cite[Theorem 1.5]{AL12a}, which shows
that a presentation of $\psi_{j, m}(Y)$ is effectively computable 
given a 
standard presentation of $Y$.  
Study of  the effect of decimation  operations on Hausdorff dimension of the images
seems an interesting topic for further research.

There are a number of directions for further generalization. 
The methods of this paper apply to arithmetic operations  applied to 
the {\em $g$-adic numbers}  for arbitrary $g \ge 2$, as defined by Mahler \cite{Mah61}.
As a topological space one has  $\ZZ_g = \prod_{p|g} \ZZ_p$. However when $g$ contains prime powers
one would use a  $g$-adic expansion corresponding to the alphabet $\sA= \{0, 1, \cdots, g-1\}$.

A second generalization is to  allow sets in the $p$-adic numbers $\QQ_p$, in which case 
addition or multiplication of arbitrary rational numbers would be permitted. 
One may also generalize the notion of $p$-adic path sets to higher dimensions, which
would correspond to $(\ZZ_p)^d$. In this case one may investigate various relaxations
of the overlap conditions imposed in the Mauldin-Williams construction. In the
real number analogue $\RR^n$ results have been obtained by Ngai and Wang \cite{NW01}
and Das and Ngai \cite{DN04}.

%
%
%

\subsection{Contents of the paper}

Section \ref{sec2} recalls Mauldin-Williams constructions, gives
two equivalent characterizations of $p$-adic path set fractals, and
determines all $Y \in \sC(\ZZ_p)$ that  contain exactly one element.
Section \ref{sec3}  gives formulas for Hausdoff dimension of path set fractals.
Section \ref{sec4} proves results on
addition of rational numbers to $p$-adic path set fractals, and on
 set-valued addition of two $p$-adic path set fractals.
Section \ref{sec5} proves results on multiplication of $p$-adic path set fractals
by rational numbers.
Section \ref{sec6} presents  examples illustrating the results.
Section \ref{sec7} makes concluding remarks about how the constructions
of this paper relate to integer matices.

\subsection*{Acknowledgments}  The authors thank the reviewer for helpful comments.
W. Abram acknowledges the  support of an
NSF  Graduate Research Fellowship. 

%
%
%

\section{Relation to  Geometric  Graph-Directed Constructions} \label{sec2}

We first describe the Mauldin-Williams geometric graph-directed construction in the real number case.
Then we  formulate a (restricted) $p$-adic analogue to it, and show that all
$p$-adic path set fractals are obtained by such a construction, and conversely.
The final subsection  characterizes those $p$-adic path set fractals containing exactly one element.

%
%
%

\subsection{Mauldin-Williams graph-directed constructions} \label{sec21}

In the 1980's Mauldin and Williams \cite{MW88} introduced general
graph-directed constructions of fractal sets over the real numbers, and computed their
Hausdorff dimensions, see also Edgar \cite[Chap. 4]{Edg08}.
We follow the notation established in 
Mauldin and Williams \cite{MW88}.

\begin{defn} \label{de21}
A {\em geometric graph-directed construction} in $\mathbb{R}^m$ consists of:
\begin{enumerate}
\item[(G1)]
 a finite sequence of nonoverlapping\footnote{Sets $J_1$ and $J_2$ {\em overlap} if their intersection
has a nonempty interior.}
, compact subsets $J_1,\ldots,J_n$ of $\mathbb{R}^m$, such that each $J_i$ has a nonempty interior,
\item[(G2)]
 a directed graph $G$ with vertex set consisting of the integers $1,\ldots,n$, such
that for each pair $(i,j)$ there is  at most one directed edge from $i$ to $j$.
Additionally, this graph must have the following properties:
\begin{enumerate}
\item[(a)]
For  each vertex $i$, there must be
at least one exit edge, i.e. some $j$ such that $(i,j) \in G$,
\item[(b)]
The underlying undirected graph must be connected.
 \end{enumerate}
\item[(G3)]
 For each graph edge $(i,j)$ there is assigned a
 similarity map $T_{i,j}: \mathbb{R}^m \to \RR^m$,  with similarity ratio $t_{i,j}$ such that:
\begin{enumerate}
\item[(a)]
 for each $i$, $\{T_{i,j}(J_j) | (i,j) \in G\}$ is a nonoverlapping family and
\begin{equation} J_i \supseteq \bigcup \{T_{i,j}(J_j) | (i,j) \in G\}
\end{equation}
\item [(b)]
 if the path component of $G$ rooted at the vertex $i_1$ is a cycle: $ [i_1,\ldots,i_q,i_{q+1}=i_1],$
then these satisfy the contraction condition
\begin{equation} \prod_{k=1}^q t_{i_k,i_{k+1}} < 1.
\end{equation}
\end{enumerate}
\end{enumerate}
\end{defn}

Note that in this construction the similarity maps $T_{i,j}$ 
will be applied to map sets  in the reverse direction to that of the edges of $G$.

Now for each $i$ let $\mathcal{K}(J_i)$ denote the space of compact subsets of $J_i$, with the Hausdorff metric, 
$\rho_H$. 
 Mauldin and Williams prove the following:

\begin{prop} \label{pr212}
For each geometric graph-directed construction, 
there exists a unique vector of compact sets, $(K_1,\ldots,K_n) \in \prod_{i=1}^n \mathcal{K}(J_j)$ such that for each $i$,
\begin{equation} 
K_i = \bigcup \{T_{i,j}(K_j) | (i,j) \in G\}.
\end{equation}
\end{prop}

\begin{proof} 
This is proved by Mauldin-Williams (\cite[Theorem 1, p. 812]{MW88}), using the results of Hutchinson \cite{Hut81}.
Note that the maps $T_{i,j}$ act in the reverse direction to that of  the edges of $G$.
\end{proof}

The {\em construction object} $K$ is then defined by 
\begin{equation} K = \bigcup_{i=1}^n K_i.
\end{equation}
The individual $K_j$ are the {\em construction sub-objects.}

Our definition of $p$-adic path sets will  
correspond to all possible construction sub-objects
 $K_v$ in the Mauldin-Williams construction.

Associated to the graph $G$ is an $n \times n$ {\emph construction matrix} $A = A(G)$ 
(with $n=|V(G)|$) given by
\begin{equation}
A=  A(G) := [t_{i,j}]_{1\le i,j \leq n},
\end{equation}
where $t_{i,j}$ is defined to be zero if $(i,j) \notin G$. Now for $\beta > 0$, set 
\[
A_\beta = [(t_{i,j})^{\beta}]_{1\le i, j \le n},
\]
and let 
\[
\Phi ( \beta):= \mbox{Spectral radius of }~ A_{\beta}.
\]
This is the largest non-negative eigenvalue of $A_{\beta}$, by the
Perron-Frobenius theorem.
Mauldin and Williams \cite[Theorem 2]{MW86} observe that 
for each construction matrix, one has 
\begin{enumerate}
\item
$\Phi(0) \ge 1$,
\item
 $\Phi(\beta)$ is a continuous, strictly decreasing
function of $\beta \ge 0$ 
\item
 $\lim_{\beta \to \infty} \Phi(\beta) =0.$
 \end{enumerate}
  It follows that there  is a unique value  $\alpha\ge 0$ such
that $\Phi(\alpha) =1$. They term this value  the {\em matrix  dimension}
of the matrix $A= A_G$.

Mauldin and Williams determine the Hausdorff dimension of the construction object $K$
and also of its individual sub-objects $K_{j}$. For the construction object $K$ it is
given as follows. 

\begin{prop} \label{prop23a}
For each geometric graph-directed construction, the Hausdorff dimension of $K$, the construction
object, is $\alpha$, where $\alpha$ is the matrix dimension of the construction matrix 
$A(G) = [ t_{i,j}]_{1\le i, j \le n}$, with $n=|V(G)|.$
That is, 
it is the unique value $\alpha \ge 0$ such that the spectral radius $\sigma(A_{\alpha})=1$,
where $A_{\beta} := [ (t_{i,j})^{\beta}]_{1 \le i,j \le n}$ for $\beta>0.$

\end{prop}

\begin{proof} This is Theorem 3 of Mauldin and Williams \cite{MW88}. 
\end{proof}

  Hausdorff dimension formulas for  construction sub-objects $K_v$ involve the strongly connected components
of the directed graph $G$, and use matrices which are  square submatrices of $A(G)$,
extracting specified 
rows and corresponding columns.
A {\em strongly connected component} of a directed graph $G$ is a maximal subgraph that is strongly
connected (i.e. each vertex in the component is reachable from each other vertex in it by a directed  path.)
We let {\em $SC(G)$} denote  the set of strongly connected components of the connected graph $G$. 
There is a natural  partial ordering on $SC(G)$
which sets  $H_1 \preceq H_2$ provided there is a directed path in $G$ from a vertex in $H_1$
to one in $H_2$.  We let $\alpha_H$ denote the matrix dimension of the square submatrix $A_H$ of 
the construction matrix $A_G$
corresponding to the strongly connected component $H$ of $G$.
\begin{prop} \label{prop23}
For each geometric graph-directed construction, the following hold.

(1) The Hausdorff dimension of $K$, the construction
object, is $\alpha$, where 
\[
\alpha = \max \{ \alpha_H | H \in SC(G)\}.
\]
 Furthermore the set
 $K$ has positive $\sigma$-finite $\mathcal{H}^\alpha$ measure. 
 
(2) The Hausdorff dimension of each construction sub-object $K_j$  is $\alpha_j$, where 
\[
\alpha_j := \max \{ \alpha_H | H \in C_j\},
\]
where $C_j$ is the set of strongly connected components of $G$ reachable from
vertex $j$.
 The sub-object $K_j$ has positive $\sigma$-finite $\alpha_j$-dimensional
 Hausdorff measure. This measure is finite if and only if
 $\{H \in C_j | \alpha_H = \alpha_j\}$ consists of (pairwise) incomparable elements
  in the partial order $\preceq$ on $SC(G)$.
\end{prop}

\begin{proof} This statement combines Theorems 4 and 5 of Mauldin-Williams (\cite[p. 814, p. 824]{MW88}).
\end{proof}

%
%
%

\subsection{$p$-adic   graph-directed constructions} \label{sec22}

We formulate a  $p$-adic variant of the Mauldin-Williams construction
inside the compact set $\ZZ_p$ as follows.

\begin{defn}\label{def22}
A {\em (restricted) $p$-adic graph-directed construction} on
the $p$-adic integers  $\ZZ_p$ consists of
\begin{enumerate}
\item[(P1)]
 a finite sequence of (identical) initial sets $J_i = \ZZ_p$, for $1 \le i \le n$;
 these sets overlap.
 \item[(P2)]
  a  finite directed labeled graph $\sG=(G, V, \sE)$ with vertex set $V$
  consisting of the integers  $1, 2..., n$,  with $\sE \subset V \times V \times \sA$, 
 with each labeled edge  assigned data $(i(e), f(e), j_e)$ in which $i(e), f(e) \in V$
denote the initial and final vertices of the directed edge, and the label
$j_e \in \sA=\{ 0, 1, ..., p-1\}$ is drawn from the  usual alphabet of $p$-adic digits.
No  two edges have the same data $(i,f, j)$. Each vertex of the
underlying directed graph $G$ has at least one exit edge.
\item[(P3)]
 to the label $j_e$ is associated a $p$-adic
similarity  map $\phi: \ZZ_p \to \ZZ_p$  given by
$$
\phi_e(y) = py+ j_e.
$$
This  is a contractive mapping in the $p$-adic metric.
\end{enumerate}
\end{defn}
\noindent Note that in  this construction  the similarity maps $\phi_{e}$ in (P3)
will be applied to  sets  in the  direction {\em reverse} to that
assigned to the directed graph edge $e$ of $G$, compare \eqref{eq221} below.

 This definition  differs from the   Mauldin-Williams real number 
 graph-directed construction in several ways.
Firstly, in  condition (P1) it  starts with initial sets $J_i$ that have overlaps, which  is forbidden
in (G1) of the Mauldin-Williams construction.  
Secondly, in  condition (P2)  the  underlying directed graph $G$ (ignoring labels) is permitted to have 
loops  ($i(e)= t(e)$) and multiple edges (having same $i(e), v(e))$, which are forbidden in (G2).
(Mauldin-Williams forbid these conditions in order to handle maps having different contraction
rations $t_{i,j}$ on different edges.)
Thirdly, condition (P3) requires that all contraction ratios $t_{i,j}$  
be equal, which is a narrower condition 
 than the Mauldin-Williams condition (G3).
Furthermore  Condition (P3) implies 
 that analogues of conditions (G3) (a), (b) automatically hold,  aside from
 the non-overlapping condition:
\begin{enumerate}
\item[(a)]
The initial sets $J_j=\ZZ_p$  satisfy  the condition 
$$
J_j \supseteq \bigcup \{ \phi_{e}(J_{f(e)} ):  e=(i(e), f(e))~~\mbox{has initial vertex}  ~i(e)=j \}.
$$
\item[(b)]
Each  map $\phi_e$ is has $p$-adic  contraction ratio $t_{e} := |p|_p = \frac{1}{p} < 1$.
Thus for   $[e_1, ..., e_q, e_{q+1}=e_1] $  a directed cycle of edges in $G$,  the contracting cycle condition holds:
\[
 \prod_{j=1}^q  t_{I(e_j), f(e_j)} < 1,
 \] 
\end{enumerate}

The following result gives existence and uniqueness for the $p$-adic construction.

\begin{thm}\label{pr221} {\em ($p$-adic Geometric Graph-Directed Construction)}
Let $\sG=(G, V, \sE)$ be a 
connected labeled graph with vertices $V = \{ 1, 2, ..., n\}$,
and with edge label alphabet
$\sA= \{ 0, 1, 2, ..., p-1\}$.
Then there exist unique nonempty compact sets $\{K_i: i \in V\}$, each contained in $\ZZ_p$,
that satisfy the set-valued functional relations, for each vertex $i \in V$, 
\begin{equation} \label{eq221}
K_i = \bigcup_{ \{ e: i(e) = i\} }\phi_e( K_{f(e)} ).   
\end{equation}
\end{thm} 

\begin{proof}
This existence and uniqueness of  the
compact set-valued fixed point (\ref{eq221}) follow from Hutchinson \cite[Theorem 3.1]{Hut81}.
The Hutchison proof establishes that the $K_i$ are obtained by the following iterative process.
We start with initial sets $K_i^{(0)} = \ZZ_p$, and iteratively define, for each vertex $i \in V$, 
$$
K_i^{(k+1)} := \bigcup_{ \{ e: i(e) = i\}} \phi_e( K_{f(e)}^{(k)} ),   
$$
We obtain a  sequence of closed sets
$$
K_i^{(0)} \supseteq K_i^{(1)} \supseteq K_i^{(2)} \supset \cdots,
$$
and these converge downwards to the compact sets $K_i,$ as $n \to \infty$.
\end{proof}
\begin{defn} \label{de26}
Any set  $Y:=K_i$ for some sub-object $K_i \subseteq \ZZ_p$ in
a (restricted) $p$-adic graph directed system is called
  a {\em geometric $p$-adic path set fractal}. We denote the
  set of all such $Y \subseteq \ZZ_p$ as $\sC_{G}(\ZZ_p)$.
 \end{defn}

In Theorem~\ref{self-sim} below we will show that this definition gives exactly the same class of
sets  as those defined  in the introduction, i.e. $\sC_{G}(\ZZ_p)= \sC(\ZZ_p)$.

Many different construction pairs $(\sG, v)$ can produce the same geometric $p$-adic path set fractal $Y$.
One can use this freedom to make good choices for $\sG$. For example, one may  require $\sG$ to
be right-resolving, and in addition to have at most one directed edge between any directed pair of
vertices, see the proof of Theorem ~\ref{th31a}(1) in Section \ref{sec3}.

%
%
%

\subsection{Path sets and $p$-adic path set fractals} \label{sec23}

We now show that geometric  $p$-adic path set fractals in Section \ref{sec22}
comprise exactly the same sets as the  images in $\ZZ_p$
of abstract path sets under a symbol labeling.
 We recall a  formal definition of {\em path set}  given in  \cite{AL12a}.
 Let $ \mathcal{A}^\mathbb{N}$ be the full one-sided shift space on $\mathcal{A}$. 
A \emph{pointed graph} 
over an alphabet $\mathcal{A}$ consists of a pair $(\mathcal{G},v)$, where $\mathcal{G}=(G, \sE)$ is a finite edge-labeled 
directed graph $G$, with labeled edges $\sE \subset E \times \sA$ 
having labels drawn from an alphabet $\sA$, 
 and $v$ a vertex of $G$. We let $V(\sG)$ and $E(\sG)$
 denote the set of vertices and directed edges of $G$, respectively.
  Following \cite{AL12a} we make a basic definition.

%
\begin{defn} \label{de211}
 For a pointed graph $(\sG, v)$ its associated
 (abstract) \emph{path set} 
 (or {\em pointed follower set}) 
   $ \sP =X_\mathcal{G}(v) \subset \sA^{\NN}$ is
the set of  all infinite one-sided symbol sequences
giving the successive labels of all one-sided infinite walks in $\mathcal{G}$ 
issuing from the distinguished vertex $v$. 
Many different $(\mathcal{G}, v)$ may give the same path set $\sP$, and we call any such
$(\mathcal{G},v)$  a \emph{presentation} of $\sP$. 
\end{defn}

Recall that  the class $\sC(\ZZ_p)$ of $p$-adic path set fractals consists of  images
$f_P(\sP) = f_p(X_{\sG}(v))$ of a path set $\sP$ under a 
digit assignment map $\bar{f}_p$, sending a path address to a $p$-adic expansion.
We show that the class $\sC(\ZZ_p)$ agrees with the
geometric class $\sC_{G}(\ZZ_p)$ given by the geometric Mauldin-Williams
construction. 
For this purpose it is helpful to know that every element of $\sC(\ZZ_p)$ has
presentation $Y := (\bar{f}_p, \sG, v)$ 
of the special form called a standard presentation in Section \ref{sec11}.


\begin{prop}\label{pr30} { \em (Standard presentation)}

(1) Every  path set $\sP = X_{\sG}(v)$ on an alphabet $\sA$ has a presentation $(\sG, v)$ that is 
right-resolving and reachable.

(2) Every $p$-adic path set fractal $Y$ in $\sC(\ZZ_p)$ has a presentation
$Y= (\bar{f}_p, \sG, v)$ that is injective, right-resolving and reachable; that is, 
a standard presentation.  Furthermore
 one may specify that  the presentation  alphabet is $\sA =\{0, 1,...,p-1\}$ with
the identity digit assignment map $\bar{i}_p$.
\end{prop}

\begin{proof}
(1) Any path set $X$ on any alphabet $\sA$ has a  right-resolving, reachable  presentation $\sP=(\sG, v)$
on this alphabet, with
$X= X_{\sG}(v)$, by a standard construction, see \cite[Theorem  3.2]{AL12a}. 

(2)  By hypothesis $Y  \in \sC(\ZZ_p) $ comes with a  presentation  $Y = f_p(X_{\sG'}(v'))$,
on an alphabet $\sA'$. Using the digit assignment map $\bar{f}_p: \sA' \to \{0, 1, ..., p-1\}$, we may 
 relabel the underlying edges of the graph of ${\sG'}$ by the image labels in  $\sA =\{0, 1, .., p-1\}$,
 obtaining a labeled directed graph $\sG''$ with 
  $Y= i_p(X_{\sG''}(v'))$. 
By (1)  the path set $X_{\sG''}(v)$ has another presentation  $X_{\sG}(v)$,
in which the new  graph $\sG$ is right-resolving and reachable, 
and uses  the same label alphabet $\sA$.
By inspection  $Y := (\bar{i}_p, \sG, v')$ is still a 
presentation of $Y$, and it  is injective, right-resolving and reachable.
\end{proof}

Note  the  one-sided shift space $\sA^{\NN}$, for $\sA= \{0,1 ,..., p-1\}$,
topologized with the product topology,  is  homeomorphic to $\ZZ_p$
with its usual $p$-adic topology, where one identifies a  symbol sequence
$\alpha_0 \alpha_1 \alpha_2 \cdots  \in \sA^{\NN}$ with the $p$-adic expansion
$$
x = \sum_{j=0}^{\infty} \alpha_j p^j = (\cdots \alpha_2 \alpha_1 \alpha_0)_p \in \ZZ_p.
$$
This is exactly the map $i_p$ underlying
the standard presentation $Y = (\bar{i}_p, \sG, v)$ of a $p$-adic path set fractal above.

Now we relate the  class $\sC(\ZZ_p)$ 
to the geometric class $\sC_{G}(\ZZ_p)$.
To this end we note  that attached to any labeled directed graph $\sG=(G, V, \sE) $ on
alphabet $\sA=\{0,1 , ..., p-1\}$ there is associated a (restricted) $p$-adic graph-directed construction,
(as in Section \ref{sec22})  based on the same graph data $\sG=(G, V, \sE) $
where the edge label $j \in \sA$ is now assigned the map
$\phi_j(x) = px +j.$

%
\begin{thm}\label{self-sim}
There holds $\sC_{G}(\ZZ_p) = \sC(\ZZ_p)$. Specifically:

(1)  Let $K_v \in \sC_{G}(\ZZ_p)$ be a construction sub-object of
a (restricted) $p$-adic graph directed
construction with data $\sG = (G, V, \sE)$
using edge maps  $\phi_{e} = px+ j_{e}$ with $0 \le j_e\le p-1$,
associated to vertex $v$ of $G$.
Create  a path
set from the same data $\sG(G, V, \sE)$, interpreting the
edge labels $j_e \in \sA =\{0,1, ..., p-1\}$,
and let $Y:= i_p( X_{\sG}(v))$
be the $p$-adic path set fractal with presentation $Y:= (\bar{i}_p, \sG, v)$.
Then $Y = K_v.$
It follows that 
  $\sC_{G}(\ZZ_p) \subset \sC(\ZZ_p)$.

(2) Let $Y \in \sC(\ZZ_p)$ be any $p$-adic path set fractal. Then it  has a standard
presentation $Y= (\bar{i}_p, \sG, v)$ with a labeled directed graph $\sG= (G, V, \sE)$,
having the additional property that   all vertices of the graph 
$G$ have at least one exit edge.
By (1) there is  an associated (restricted)
$p$-adic graph directed construction having 
a  sub-object $K_v$, with $Y=K_v$. Thus
 $\sC(\ZZ_p) \subset \sC_{G}(\ZZ_p)$.
\end{thm}

\begin{proof} 
(1) The correspondence between $K_v$ and $Y_v$
proceeds by relating  paths to the address labels of points in the 
graph-directed fractal, compare Edgar \cite[Section 4.3]{Edg08}.
We study the set-valued  iteration $K_i^{(k)}$ given in 
 Theorem  \ref{pr221}
 for the geometric $p$-adic path set fractal determined by $\sG=(G, V, \sE)$.
 We prove by induction on $k \ge 0$ that for all vertices $i$ 
 $$
 K_i^{(k)}: =\bigcup  \Big( (\alpha_0 + \alpha_1 p + \cdots +\alpha_{k-1} p^{k-1})+ p^k \ZZ_p\Big)
 $$
 where the set union is taken over label sequences 
$ (\alpha_0, ..., \alpha_{k-1})$ of legal walks in the directed graph $\sG$
of length $k$ starting from vertex $i$. (The edge leaving the initial vertex $i$ has label $\alpha_0$.)
 The hypothesis that an exit edge exists
from each vertex guarantees that all paths extend one step. The base case $k=0$
holds since all $K_i^{(0)} = \ZZ_p$. For the induction step, we have
$K_i^{(k+1)}$ is comprised of sets
$$
\phi_{e} (K_{f(e)}^{(k)}) = j_e + \sum_{i=0}^{k-1} \alpha_i p^{i+1} + p^{k+1}\ZZ_p
$$
where $(\alpha_0, ..., \alpha_{k-1})$ are labels from a directed walk in $\sG$ starting
from vertex $f(e)$. But now $(j_e, \alpha_0, ..., \alpha_{k-1})$ are vertices of a directed
walk of length $k+1$ starting from vertex $i$, and all such walks are enumerated this
way. This completes the induction step. Letting $k \to \infty$, for each vertex $i$ these sets decrease
to $K_i$, which is now identified with all infinite walks in $\sG$ starting from vertex $i$. Choosing
$i$ to be the original marked vertex, we obtain $K_i = \bar{i}_p (X_{\sG}(v)) = Y_v$, as asserted.

(2) Given a $p$-adic path set fractal $Y$ we take a standard form presentation 
$Y= (\bar{i}_p, \sG', v)$.  We now prune the graph $\sG'$ to remove any vertices with
no exit edges,  since such vertices contribute no infinite paths to the path set $\sP= X_{\sG}(v)$, and 
leave $Y$ unchanged. The new graph may still have vertices with no
exit edges, but by repeating  this operation a finite number of times, we will arrive at a 
presentation
$Y=(\bar{i}_p, \sG, v)$ in which all vertices have at least one exit edge. The right-resolving
and reachability properties are unaffected, so the new presentation is still standard.
The construction of part (1) now applies to give the result.
\end{proof}

It is known  that the class of  path sets on a fixed alphabet is closed
under finite unions and intersections ( \cite[Theorem 1.1]{AL12a}).
Theorem ~\ref{self-sim} implies 
that the collection of  $p$-adic path set fractals $\sC(\ZZ_p)$ is 
closed under set union and intersection as well.

%
%
%

\subsection{One element $p$-adic path set fractals}\label{sec24a}

We characterize path sets consisting of a single element.

%
\begin{thm}\label{th15} {\em (Single element $p$-adic path set fractals)}
The  $p$-adic path set fractals $Y \in \sC(\ZZ_p)$ that consist of a single element
are exactly those $Y= \{ r\}$  for which $r$ is a $p$-integral rational number, i.e.
$r \in \QQ\cap \ZZ_p$.
\end{thm}

\begin{rem}
This simple  result  supplies a dynamical characterization 
of the $p$-integral rational numbers $r$ inside $ \ZZ_p$.
\end{rem}

\begin{proof}
Given a  $p$-adic path set fractal $Y$, 
assume it is given with a standard presentation $(\bar{i}_p, \sG, v_0)$.
Using the pruning construction used in proving Theorem \ref{self-sim}
we may without loss of generality assume each vertex in $\sG$
has at least one  exit edge. 
Such a presentation has an underlying path set $X_{\sG}(v_0)$
consisting of a single infinite path if and only if 
there is exactly one exit edge from each vertex, and if the path is eventually
periodic. The latter forces any element $r$ to be a rational number in $\ZZ_p$.
Conversely, we may easily construct a path set consisting of a single
element giving the $p$-adic expansion of $r$.
\end{proof}

%
%
%

\section{Hausdorff Dimension  of $p$-adic path set fractals} \label{sec3}

We obtain a formula for the Hausdorff dimension of a $p$-adic path
set fractal $Y$, computable from a standard form presentation of $Y$.
The formula is based on  a Hausdorff dimension
relation between  $p$-adic path set fractals  and graph-directed constructions
on the real numbers.

To state the result, we
note  that the  {\em adjacency matrix} $A= A(G)$ of a directed graph $G$ is
a non-negative integer matrix whose rows and columns are numbered by the vertices of $G$ (in the same
order) with entry $A_{ij}$ counting the number of directed edges outgoing from vertex $i$ and incoming to
vertex $j$. 

\begin{thm}\label{th31a} { \em (Hausdorff dimension formula)}
Let $Y$ belong to $\sC(\ZZ_p),$ and suppose that $Y:=(\bar{f}_p, \sG, v)$
is a standard form presentation of $Y$.

(1)  The map $\iota_p: \mathbb{Z}_p \rightarrow [0,1] \subset \RR$ sending 
$\alpha= \sum_{k=0}^{\infty}{\alpha_k p^k} \in \mathbb{Z}_p$
to the corresponding real number with base $p$ radix expansion
 $$\iota_p(\alpha) :=\sum_{k=0}^\infty \frac{\alpha_k}{p^{k+1}}$$
is a continuous map.  Under this map
the  image set  $K_0:= \iota_p(Y) \subset [0,1]$  is a
construction sub-object 
 of a Mauldin-Williams graph-directed fractal 
whose edge similarity maps are all  contracting
similarity maps of $\RR$ with contraction ratio $1/p.$

(2) The Hausdorff dimensions of these sets are related by
$$
d_{H}(Y) = d_{H}(K_0).
$$

(3) The Hausdorff dimension 
\begin{equation}\label{eq32}
d_{H}(Y) = d_{H}(K_0) = \log_p \alpha,
\end{equation}
where $\alpha= \sigma(A(\sG))$ is the spectral radius of the adjacency matrix $A= A(\sG)$ of 
the underlying directed graph $G$ of $\sG$.
\end{thm}

\begin{rem} 
(1) In  Theorem 1.9 of  \cite{AL12a} 
the topological entropy of a path set $X= X_{\sG}(v)$,
with a right-resolving reachable presentation  is given by
$$
h_{top}(X) = \log \sigma(A(\sG)),
$$
where $ \sigma(A(\sG))$ is the spectral radius as above. We deduce that
for any $p$-adic path set fractal $Y$ constructed from $X$ by
an injective presentation  $Y:= (\bar{f}_p, \sG, v)$ ,  its Hausdorff dimension is  
$$
d_{H}(Y) = \frac{h_{top}(X)}{\log p}.
$$
This extends
to path sets  a result that is well known  in the shift-invariant case (Furstenberg \cite[Proposition III.1]{Fur67}).

(2) The standard presentation assumption for $Y$ above
is needed  to guarantee equality of the
Hausdorff dimension with $\log_p \alpha$. For a general presentation $Y:=(\bar{f}_p, \mathcal{G}, v_0)$
of $X$ the  adjacency matrix counts the growth rate of number of paths,
which upper bounds the number of distinct sequences of path labels. 
That is, one always has
$$ 
d_{H}(Y) \le \frac{ h_{top}(X_{\sG}(v))} {\log_p}  \le \log_p \sigma(A(\sG)).
$$

(3) The allowed values $\sigma (A(\sG))$ that may occur above are exactly
the  class of positive real algebraic integers called {\em Perron numbers},
  introduced by Lind \cite{Lin84}.
\end{rem}
\begin{proof}
(1) The map $\iota_p: \ZZ_p \to  [0,1] $ is surjective and one-to-one away from a countable
set. It is continuous because the $p$-adic topology is strictly finer than the comparable topology
on base $p$ expansions of real numbers.

We are given a standard presentation of $Y= ( \bar{i}_p, \sG, v)$, where without loss
of generality the alphabet $\sA=\{0, 1, ..., p-1\}.$ The graph $\sG$ is right-resolving
and reachable, but this will not be sufficient to obtain a Mauldin-Williams construction
for the image of $X_{\sG}(v)$ preserving  symbol sequences. We need a standard
presentation with extra properties. 
We call  a presentation  {\em right-separating}  if 
the underlying directed graph $G$ of $\sG$
has no multiple edges. \smallskip

{\bf  Claim.} {\em There exists a presentation $Y=(\bar{i}_p, \sG, v)$  in which $\sG$  is both right-resolving 
and right-separating.}\smallskip

 To show the claim, given  a right-resolving presentation $(\sG', v')$, we  show it may be converted
to a right-separating presentation  by making use of a vertex-splitting construction, as
follows. Suppose that a vertex $v$ of $\sG'$ has $k \ge 2$ labeled edges
from a vertex $w$, which necessarily has distinct labels. 
We  create a new labeled graph $\sG''$ that retains all vertices of $\sG$ except
 $v$ and replaces $v$ by $k$ vertices $v^{(i)}$, $1 \le i \le k$. 
 
 {\em Case 1.} $ w\ne v$.
 
 In this case we assign a 
 single new labeled edge from $w$ to each $v^{(i)}$
 such that  as $i$ varies, the corresponding edge labels  exhaust  the $k$ labels of edges from $w$ to $v$. 
 The exit edges assigned each $v^{(i)}$ 
 each duplicate the exit edges from $v$, both in multiplicity and in labels, with self-loops 
 of $v$ corresponding to self-loops of $v^{(i)}$.
 The entering edges to $v^{(1)}$ will all be the same in multiplicity and in labels as for $v$, while for $v^{(i)}$ with
 $i\ge 2$ there will be no entering edges from the rest of the graph, with the exception of self-loops,
 assigned as above.
 Finally all  edges between any two vertices distinct from the $v^{(i)}$ will be the same as
in the original graph.

{\em Case 2.} $w= v$
 
 In this case $v$ has $k$ self-loops, which have distinct labels since the right-resolving property
 is assumed. We may identify the $k$ loop labels with $\{ 1, 2, ..., k\}$ in some fixed fashion,
 and then assign a directed edge from $v^{(i)}$ to $v^{(j)}$ with edge label corresponding to
 $i+j ~(\bmod~k)$, for $1 \le i, j \le k$. The other exit edges assigned each $v^{(i)}$ will duplicate the exit
 edges of $v$, in multiplicity and labels (excluding self-loops). The entering edges of $v^{(1)}$
 will  be the same in multiplicity and labels as for $v$ (excluding self-loops). All other $v^{(i)}$
 have no entering edges  coming from any other  vertices of the original graph $v' \ne v$.

  Now that $\sG''$ is constructed, we assert  that  the path sets from all states
$(\sG'', w')$ for  $w' \ne v^{(i)}$ agree with those of $(\sG', w')$, while all path sets $(\sG'', v^{(i)})$
agree identically with that of $(\sG', v)$.  This assertion  may be established  by viewing $\sG'$ as a covering of
$\sG$ which preserves edge labels, which has the $k$ vertices $v^{(i)}$ lying above vertex $v$,
 and all other vertices agreeing.
One may check that each edge in $\sG'$  from 
a given initial vertex  $v'$ lifts uniquely  to a suitable vertex and edge above it in $\sG''$
(here $v'=w$ is the only interesting case), except that self-loops from $v$ lift to a
self-loop for any initial vertex $v^{(i)}$. After the first step, any path lifts uniquely to $\sG''$.  
 Conversely any labeled path in the lifted graph projects to an allowable
labeled path in $\sG'$. The assertion follows.

This construction has the feature that the  new graph $\sG''$ is still right-resolving. 
In consequence the  construction may be repeated. 
In doing so,  we must eventually arrive at a
right-resolving presentation that is also   right-separating. To see this, 
assign to each vertex an integer invariant that is the product of the multiplicities of
all entering edges. When a vertex is split, this invariant decreases for all of the $k$ descendants
$v^{(i)}$, and remains the same for all other vertices of the graph.
 By the well-ordering of $\NN$, the splitting procedure will eventually halt
 at a right-separating presentation. This establishes the claim.\smallskip

We have now obtained a standard presentation that
is also  a  right-separating presentation $(\bar{i}_p, \sG, v)$ of $Y$. 
By pruning vertices with no exit edges (repeating the operation finitely
many times, as necessary), we may obtain such a presentation 
in which additionally each vertex has at least one exit edge.
Associated to $\sG$ are $|V(\sG)|$ path sets $X_{\sG}(\bw)$,
$\bw \in V(\sG)$, and corresponding $Y_{\bw} := i_p(X_{\sG}(\bw)) \in \sC(\ZZ_p)$.

We now proceed to  map these sets to real image sets which are 
corresponding graph-directed constructions.
For convenience we re-number the vertices of $\sG$, $0 \le i \le n$,
where $n = |V(\sG)|-1$, with vertex $0$ corresponding to the original $v$.
We  map the individual sets $Y_{j}$ under the map $\iota_{p}+2j$ to the  image sets
 $$
 K_{j} := \iota_p(Y_{j}) + 2j \subset [2j, 2j+1].
 $$
 The integer shifts by $2j$ make all sets $K_j$ disjoint in $\RR$; Thus $K_0 = \iota(Y).$

We  show below that  the sets $(K_0, K_1, \cdots , K_n)$ are 
the complete set of construction sub-objects
of a particular Mauldin-Williams graph-directed construction.
The integer shifts made above enforce the non-overlapping condition needed in that construction.

  We  set up a   Mauldin-Williams geometric graph-directed construction in $\RR$,
which has  construction object $K$  contained in the compact set  $[0,2n]$, for which 
the sets $(K_0, K_1, ..., K_n)$ will form  the construction sub-objects.
The initial sets will be $J_j = [2j, 2j+1]$ for $0 \le j \le n$; they satisfy 
the non-overlapping property (G1). It uses the same directed graph $G$ as that of $\sG$.
The graph $G$
has all the correct properties (G2) to be a graph in
the Mauldin-Williams construction: it is connected, has at most one directed edge between any
ordered  pair of vertices (by the right-separating property), and  each vertex has at least one exit edge. 
To each directed labeled edge  $e = [ i_1, i_2]$ of $G$
with   label $j_e$ and map $\phi_{e}$ we associate the real-valued map  
\[
T_{e}(x) =T_{i_1, i_2}(x) := \frac{1}{p} \Big((x-2i_2)+j \Big)+ 2i_1.
\]
This map is a similarity  with contraction ratio $1/p$, and note that
\begin{equation}\label{inclusion}
T_{i_1, i_2}(J_{i_2}) \subset J_{i_1}.
\end{equation}
Now condition (G3)(a),  that 
the sets $\{ T_{i,j}(J_j) : e=(i,j)\}$ are non-overlapping for each $i$,
holds as a consequence of the  right-resolving
property of $\sG$. The second condition 
\[
 J_i \supseteq \bigcup \{T_{i,j}(J_j) | (i,j) \in G^{*}\}
\]
follows from (\ref{inclusion}). Finally condition (G3)(b) holds since every map $T_{e}(x)$ is
a strict contraction.

By the basic theorem of Hutchinson \cite[Theorem 3.1]{Hut81},
 this construction has a unique compact attractor $K$ consisting
of a collection of disjoint compact sub-objects $\{ K_j: 0 \le j \le n \}$, and it remains to verify
that
$$
K_0 = \iota_p(Y).
$$
The sets $K_j$  are 
obtained  by the Mauldin-Williams (iterative)
geometric graph-directed construction using \cite[Theorem 1]{MW88},
starting with the (disjoint) initial sets $J_{j} := [2j,2j+1]$. 
After $m$-iterations, we have sets $J_{j}^{(k)}$ which for fixed $j$ form
nested sequences of compact sets, each having nonempty interior. 
We  then  obtain the limit objects
$K_j := \bigcap_{k} J_{j}^{(k)}$. 
The Mauldin-Williams construction object is $K = \cup_{0 \le j \le n} K_{j}.$
One can
prove by induction on $k$ that
\[
J_{i}^{(k)} = 2j + \bigcup \Big( \frac{\alpha_0}{p} + \frac{\alpha_1}{p^2} + \cdots + \frac{\alpha_{k-1}}{p^k} +\frac{1}{p^k}[0,1]\Big).
\]
where the set union runs over all symbol sequences of length $k$ on the labeled directed graph $\sG$
starting from  vertex $i$.
From this construction one sees that  $K_{0} = \iota_p(Y)$ 
since $K_{0} \subset [0,1]$ and the underlying symbol sequences of $Y$ and of $K_{0}$
agree.  Moreover one sees that $K_{j} \subset [2j,2j+1]$ and all
 the construction sub-objects satisfy
$K_j = (\iota_p) (Y_j)+2j, ~~~0 \le j \le n.$

(2)   The definition of $p$-adic Hausdorff dimension is quite similar 
to Hausdorff dimension for real numbers on the interval $[0,1]$, cf. Abercrombie \cite{Ab95}.
An {\em  $\epsilon$-covering} of $Y \subset \ZZ_p$ is a covering of $Y$
by a countable collection of $p$-adic open balls all having diameter at most $\epsilon$.
and considers the quantities
$$
m_{\beta}(Y) := \lim_{\epsilon \to 0}\,\Big(  \inf_{\epsilon -cover} (\mbox{Vol}(B(x_i, \epsilon_i)))^{\beta} \Big),
$$
in which the data $\{ (x_i, \epsilon_i): i \ge 1\}$ describes the covering, specifying  center $x_i$ and radius $\epsilon_i$
of $p$-adic disks,
with all $0 < \epsilon_i \le \epsilon$, and 
$\mbox{Vol}(S)$ denotes the usual $p$-adic measure of $S \subset \ZZ_p$.
There is a cutoff value $\beta_0$ such that $m_{\beta}(Y) = \infty$ for $\beta> \beta_0$
and $m_{\beta}(Y) =0$ for $\beta< \beta_0$; this is the Hausdorff dimension of $Y$.
We use the following basic fact, following  \cite[Section 3.2 ]{Lag09}.

{\bf Claim.} {\em The  
 mapping
$\iota_p: \ZZ_p \to [0,1]$ which sends a $p$-adic number $\lambda= (  \cdots \alpha_2\alpha_1\alpha_0)_p$ 
to the real number with base $p$ expansion $.\alpha_0\alpha_1\alpha_2 \cdots$ is continuous
and one-to-one off a countable set.
This mapping preserves Hausdorff dimension,  i.e $d_{H}(Y) = d_{H} (\iota_p(Y)).$}

To verify the claim, note that one can expand each set in  a  $p$-adic covering  of a set $Y$
to a  closed-open disk
$$
B(m, p^j) = \{ x \in \ZZ_p: ~x \equiv m~(\bmod~p^j)\}
$$
(which has diameter $\frac{1}{p^j}$), with at most a factor of $p$ increase in
diameter, and similarly one can inflate any real covering to a covering
with ternary intervals $[\frac{m}{p^j}, \frac{m+1}{p^j}]$ with at most a factor of
$p$  increase in diameter. But these special intervals are assigned the same diameter under
their respective metrics, and this can be used to show the Hausdorff dimensions
of $Y$ and $\iota(Y)$ coincide.  (The Hausdorff measures of the resulting 
sets is not proved to coincide by this argument.) 

The truth of the  claim immediately yields $d_{H}(K_0) \equiv d_{H}(\iota(Y)) = d_{H}(Y)$, as asserted.

(3)  We are given a standard presentation $Y=(\bar{i}_p, \sG, v_0)$ of the $p$-adic path set fractal
$Y$.

Assume first that this presentation  is  right-separating.
In that case we can directly  apply the formulas  of Mauldin-Williams to the construction made
in (1). The set $K_{v_0}$ is connected by directed paths to every vertex of the graph $\sG$, so
that the vertex set $C_1 = SC(\sG)$ consists of all strongly connected components of $\sG$.
By Proposition \ref{prop23}
 the Hausdorff dimension of the sub-object $K_{v_0}$
is then the same as that of the full construction object $K= \bigcup_{v} K_{v}.$

The Hausdorff dimension of the full object $K$ can now be computed
using Proposition \ref{prop23a}.
In our case all nonzero  maps for $G$ are similarities  with constant ratio $1/p$,
which  yields the formula for the scaled construction matrix
$$
A_{\beta} =  [ t_{i, j}^{\beta}]_{1 \le i, j \le n} = (\frac{1}{p})^{\beta} A_G,
$$
in which $A_G$ is the {\em adjacency matrix} of the directed graph $G$,
given by
\[
A_G = [ m_{i,j}]_{1 \le i, j\le n}
\]
where  $m_{i,j}$ counts the number of directed edges from vertex $v$ to vertex $w$.
Now set
$$
\Phi(\beta) := \mbox{ Spectral radius of}~~A_{\beta},
$$
and the special form of $A_{\beta}$ yields
$$
\Phi(\beta)= 
 \lambda_0 p^{-\beta},
$$
in which  $\lambda_0$ is the spectral radius $\sigma(A_G)$.
 By  Proposition \ref{prop23a} the 
full construction object has Hausdorff dimension
$$
d_{H}(K) = \alpha
$$
where $\Phi(\alpha) =1$. This requires
$$
\alpha = \log_{p} \lambda_0= \log_p \sigma(A_G),
$$
the desired formula.

It remains to  treat the general case, in which the intial standard presentation $Y= (\bar{i}_p, \sG, v_0)$ is not necessarily right-separating. We show  that the formula for Hausdorff dimension continues to hold.
To handle this case, we study the effect of the state-splitting construction introduced earlier  to
convert a right-resolving presentation $(\sG', v')$ of a path set to a presentation $(\sG, v_0)$
that is also right-separating. It suffices to  show that every step of this procedure yielding a graph $(\sG'', v')$ preserves
the value of the spectral radius (Perron eigenvalue) of the associated 
nonnegative integer matrix $A_{\sG''}$. If this is shown, then the already proved formula for
the Hausdorff dimension for the $p$-adic path set associated to $(\sG, v_0)$ will carry
over to that for $(\sG', v')$.

To check  that the spectral radius is preserved under this operation,
we use the known fact that 
the spectral radius of a nonnegative matrix $A$ is
given by
$$
\sigma (A) = \lim_{k \to \infty}   (N_k(A))^{1/k},
$$
in which $N_k(A) = {\bf e}^T A^k {\bf e}$, where ${\bf e} = [ 1, 1, .., 1]^T$ is a column vector.
Here $N_k(A)$ counts the number of directed paths of length $k$ between all pairs
of vertices of $A$. (The existence of the limit is part of the assertion.)
We use the  fact that $\sG''$ is a covering of $\sG'$ and that all (labeled) paths of $\sG'$
lift uniquely to paths of $\sG''$, with the exception of paths that have starting vertex $v$,
which have $s$ distinct lifts, where $s$ was the number of vertices of $\sG$
that were split. From this we conclude that
$$
N_k(A_{G'}) \le N_k(A_{G''}) \le s N_k(A_{G'}).
$$
Since $s$ is constant, we conclude that
$$
\sigma(A(G')) \le \sigma(A_{G''}) \le  \lim_{k \to \infty}  s^{1/k}  N_k(A_{G'})^{1/k}= \sigma(A_{G'}),
$$
giving the result. 
\end{proof}

\begin{proof}[Proof of Theorem \ref{th12}]
This  result is immediate from Theorem \ref{th31a},
combining  (2) and (3).
\end{proof}

%
%
%

\section{$p$-adic  addition  and path set fractals} \label{sec4}

We analyze the effect on $p$-adic path set fractals of addition
of $p$-integral  rational numbers $r \in \QQ$, viewing $\QQ$ as a subfield of $\QQ_p$. 
We describe algorithms which when given a presentation $(\sG, v_0)$ 
of a path set $Y$, will produce a presentation of $Y+r$.

%
%
%
\subsection{Sum of a path set and  a $p$-integral rational number}\label{sec41}

 Theorem \ref{th13} is an immediate corollary of the following
 stronger result. Recall that a {\em $p$-integral} rational number $r$
 is any $r \in \QQ \cap \ZZ_p$.
 
 \begin{thm} \label{th41}
 Let $Y$ belong to $\sC(\ZZ_p)$, and suppose it  has a standard presentation
 $Y:= (\bar{i}_p, \sG, v_0)$ having $V$ vertices. Suppose also that  $r$ is a $p$-integral rational
 number,
 which has a  $p$-adic expansion with  pre-periodic part of length
 $Q_0$ and a periodic part of period $Q$. Then
 the additively shifted set $Y^{'} := Y + r \in \sC(\ZZ_p)$, and it has a 
 right-resolving presentation having at most
 $ 2p(Q_0 + Q)V$ vertices.
 \end{thm}
 
 \begin{proof}
 We give an explicit construction of a presentation for $Y + r$
 which  certifies it is a $p$-adic path set fractal, starting from a given
 standard presentation.
 
 Suppose first that we are in the special case where
 $r$ has a purely periodic $p$-adic expansion $(\cdots c_2 c_1 c_0)_p$ , of period $Q$, with $c_{j+Q} =c_j$,
 and write
 $$
 r = \sum_{j=0}^{\infty} c_j p^j = \big( \sum_{j=0}^{Q-1} c_j p^j \big)\big(\sum_{k=0}^{\infty} p^{kQ} \big).
 $$
 
 We aim to construct  a 
 standard presentation $Y' = (\bar{i}_p, \sG', \bw_0)$ and show that $Y' = Y + r$.
 (At the level of symbols we may identify $Y'$  with the underlying path set $\sP'=(\sG', \bw_0)$ since the
 identity map matches them.)
  The vertex states of $\sG'$ will be labeled $\bw= (v, f, e, a)$,
 in which:
 \begin{enumerate}
 \item[(i)]
  $v$ denotes a vertex of $\sG$;
  \item[(ii)]
   $f$  denotes a place-marker
 in the periodic portion of the $p$-adic expansion of $r$, and saisfies $0 \le f\le Q-1$;
 \item[(iii)]
  $e$ keeps
 track of the current amount of  carry-digit information not yet incorporated
 in the sum-set $p$-adic expansion, 
 \item[(iv)]
 $a$ with $0 \le a \le p-1$ denotes an
 edge label value. 
 \end{enumerate}
 The initial vertex will be $\bw_0= (v_0, 0, 0, 0).$
  We will establish the upper bound $e \le 2$ on the maximum size  of a carry-digit   in the analysis below.

 The exit edges of $\sG'$ map a vertex $\bw$ to $\bw' =(v', f', e',a')$ in which
 there is a directed labeled  edge $(v, v') \in \sG$, with label $\ell_1$
 satisfying $\ell_1=a'$, 
 and the value $f' \equiv f+1 ~(\bmod\, Q)$. 
 The edge label $\ell^{'}$ assigned to this edge will be $0 \le \ell^{'} \le p-1$ with
 \begin{equation}\label{l-value}
 \ell^{'} \equiv  e+\ell_1 + c_f ~(\bmod \, p),
 \end{equation}
 and the value $e'$ is required to satisfy
 \begin{equation}\label{det-value}
 e' = \frac{1}{p} \big( e + \ell_1+ c_f -\ell^{'}).
 \end{equation}
 Finally we  define the graph $\sG'$ to consist of all states reachable from the initial
vertex $\bw_0$, and all edges constructed between these states.  

We first show that all reachable vertices
satisfy the carry-digit bound $e \le 2$; this shows that the graph $\sG'$ is finite and
also bounds its size. The carry-digit bound is proved by induction on the number
of steps $n$ along a directed path. The base case $n=0$ has $e=0$.
For the induction step, using the rule above
$$
e' = \frac{1}{p} \big( e + \ell_1+ c_f -\ell^{'} ) \le \frac{1}{p} \big( 2 + (p-1) + (p-1) \big) \le 2,
$$
completing the induction step. We conclude in this case that $\sG'$ has
at most $2p(Q-1)V$ vertices.

We next show that the presentation $Y'= (\bar{i}_p, \sG', \bw_0)$ is a standard presentation.
We first check that $\sG'$ is right-resolving.
To see this, note that the exit edges from a vertex $\bw$ correspond to exit edges from
vertex $v$ in the right-resolving graph $\sG$, whence any two edges
have different values of $\ell_1$.  Now the exit edge label $\ell^{'}$ is an invertible linear
function of $\ell_1$ by (\ref{l-value}), since the values $e$ and $c_f$ are fixed by
$\bw$, so all exit edges have distinct labels, as asserted. This graph $\sG'$ is 
reachable from vertex $\bw_0$ by construction, so we have a standard presentation.

We next observe that  a  lifted path in $\sG'$ uniquely determines the  path in $\sG$ it lies over.
 This follows since the path label value $\ell_1$ is uniquely recoverable from the path label $\ell^{'}$
and the vertex data on $\sG$, using (\ref{det-value}), since $e'$ is known and $c_f$ is known from
the vertex label $f$. The underlying path on $\sG$ determines the quantity
$$
y_{n} = \sum_{k=0}^{n-1}  a_{k} p^k,~~~i=1,2,
$$
corresponding to the initial part of the $p$-adic expansion of a value $y \in Y$ being
determined by the $\sG$-path. (The values $a_k$ are the successive labels $\ell_1$ along the $\sG$-path.)
Conversely,  each path in $\sG$ with initial vertex $v_0$  lifts to a unique path in $\sG'$
with initial vertex $\bw_0$.
Given a vertex $\bw$,  a labeled edge $(v, v')$ with label $\ell_1$
determines the values $a'= \ell_1$ and $e'$ a unique vertex $\bw'$ that $\bw$ connects to. 

We now show  that for an infinite  path in $\sG$ determining $y \in Y$,
the labels of the lifted path in $\sG'$ suffice to compute the associated value $y' \in Y'$.
We prove  this by induction on $n$, for the $n$-step initial path.
The  successive edge labels $\{\ell_i^{'}: 0 \le i \le n-1\}$
of the lifted path in $\sG'$  with the end vertex data $e'=e_{n}$ determine the quantity
$$
y_n^{'} := \sum_{k=0}^{n-1} \ell_k^{'} p^k  + e_n p^n.
$$
We establish by induction on $n$ that
\begin{equation} \label{match}
y_{n-1}^{'} = y_{n-1} + r_{n-1},
\end{equation}
in which
$$
r_{m} := \sum_{j=0}^{m-1} c_k p^k,
$$
is a truncated version of the $p$-adic expansion of $r$.
The base case $n=1$ is clear.
  By the induction hypothesis, 
 \begin{eqnarray*} 
 y_{n} +r_{n} &= &(y_{n-1}+r_{n-1}) + a_{n-1} p^{n-1}+ c_{n-1} p^{n-1}\\
 &=&\big( \sum_{k=0}^{n-2} \ell_k^{'}p^k + e_{n-1}p^{n-1}\big) + a_{n-1} p^{n-1} + c_{n-1}p^{n-1} \\
 &=& \big( \sum_{k=0}^{n-1} \ell_k^{'}p^k \big) + e_n p^n,
\end{eqnarray*}
 the last equality holding  by virtue of (\ref{det-value}), using $c_f= c_{n-1}$ and $\ell_1 =a_{n}$.
 This completes the induction step, proving (\ref{match}).
  
  Now the lifted path data  yields  the $p$-adic limit
  $$
 \lim_{n \to \infty} \, \sum_{k=0}^{n-1} \ell_k^{'} p^k  =  \lim_{n \to \infty} (y_{n} + r_{n})- \lim_{n \to \infty} e_n p^n =y+r.
 $$
 We conclude that the lifted path of $\sG'$ corresponding to $y \in Y$
  determines   the point $y':=y+ r \in Y'$. It follows that $Y' = Y + r$, as asserted.
  This completes the argument in the case that  $r$
  has a purely periodic $p$-adic expansion.
 
 It remains to 
 treat the general case  where $r$ has a preperiodic part, say 
 length $Q_0$. We must extend the construction above and  upper bound the number of states in the constructed 
 presentation $(\sG, \bw)$. The extension is routine:
 we add extra vertices $\bw := (v, d_j, e, a )$ to $\sG$,
 in which $v$ denotes a vertex of $\sG$, $d_j$ marks the $j$-th preperiodic digit of $r$,  $1 \le j \le Q_0$.
 The exit edges of $\sG'$ map a vertex  to $\bw' =(v', d_{j+1}, e',a')$ in which
 there is a directed labeled  edge $(v, v') \in \sG$, with label $\ell_1$
 satisfying $\ell_1=a'$. 
  The edge label $\ell^{'}$ assigned to this edge will be $0 \le \ell \le p-1$ with
 \begin{equation}\label{l-value}
 \ell^{'} \equiv  e+\ell_1 + d_{j} ~(\bmod \, p),
 \end{equation}
 and the value $e'$ is required to satisfy
 \begin{equation}\label{det-value}
 e' = \frac{1}{p} \big( e + \ell_1+ d_j -\ell^{'}).
 \end{equation}
The final preperiodic digit exit edges go to vertices $\bw:=(v', f_0, e', a')$ in the earlier set.

It is straightforward to check that the underlying labeled graph has the 
right-resolving property. Next  one must check this extension 
preserves the  lifting property of paths, we omit  the details.

Finally we must  upper bound the total number of vertices in the graph $(\sG, \bw)$.
 One finds that the  preperiodic part contributes at most $2Q_0 V p$ vertices, and
 the periodic part contributes at most $2Q Vp$ vertices. 
 \end{proof} 
 
 \begin{rem}
The key features in  this proof are: (i) the $p$-adic carry digits propagate to higher powers of
 $p$ and do not disturb earlier $p$-adic digits; (ii) the size of the carry digits is bounded above.
 Property (i) fails in  real number arithmetic, and there is no real number analogue of 
 this result.
\end{rem}
%
%
%
\subsection{Minkowski sum of two $p$-adic path sets}\label{sec42}

We  show that the Minkowski sum 
of two $p$-adic path set fractals is itself a path set,
establishing Theorem  \ref{th14}. This proof  is constructive, but it no
longer produces a right-resolving presentation. 

\begin{proof}[Proof of Theorem \ref{th14}.]
We suppose that $Y_1 :=(\bar{i}_p, \sG_{1}, v_1)$ and $Y_2:= (\bar{i}_p, \sG_2, v_2)$ come
with 
standard presentations.
We use these presentations to directly construct a  presentation $Y':= (\bar{i}_p, \sG_{1,2}, \bw_0)$,
which is not necessarily standard, and show that $Y'= Y_1+Y_2$, the Minkowski sum, 
 certifying membership in $\sC(\ZZ_p)$.
  
To begin the construction,  $\sG_{1, 2}$ will have vertices labeled  $\bw:=(v_{j, 1}, v_{k, 2}, e, a)$
where $v_{j,1} \in V(\sG_1), v_{k,2} \in V(\sG_2)$, and $e \ge 0 $ is an integer encoding 
 carry-digit information, and $0 \le a \le p-1$ specifies an allowed edge entry label in $\sG_{1,2}^{+}$.
  
The exit edges from a given vertex $\bw$ go to a new vertex $\bw'= (v_{j',1}, v_{k',2}, e',a')$ in which

(a) there is a directed edge of $\sG_1$ from $v_{j,1}$ to $v_{j', 1}$ having label $\ell_1$
satisfying $\ell_1=a'$;

(b) a directed edge of $\sG_2$ from $v_{k,2}$ to $v_{k', 2}$ with label $\ell_2$, and

(c) the constructed edge is assigned the label $\ell$, $0 \le \ell \le p-1$, determined by
\[ 
\ell \equiv  e + \ell_1 + \ell_2~(\bmod ~p), ~~~0 \le \ell \le p-1
\]

(d) the  new carry-digit is
\[
e' = \frac{1}{p} \big( e + \ell_1 + \ell_2 - \ell) \ge 0.
\]

The initial pointed vertex of the graph $\sG_{1,2}$ is $\bw_0:=(v_{0,1}, v_{0,2}, 0,0).$
We now define $\sG_{1,2}$ to consist of all vertices above reachable 
from $\bw_0$ by  some directed
path.  We show this is  a finite graph by establishing that  that the ``carry-digit" 
in any reachable vertex satisfies $e \le 2$. This follows by induction on the length of 
the path. The base case is $n=0$ where $e=0$. For  the induction step,
we upper bound the new value of $e$ via
$$
e' = \frac{1}{p} \big(e+ \ell_1+ \ell_2 - \ell) \le \frac{1}{p} \big( 2 + (p-1) + (p-1)\big) \le 2,
$$
completing the induction step. We then insert all edges between these vertices produced
in the construction above.

To see that $Y_1^{'} = Y_1+ Y_2$, we prove
by induction on $n \ge 0$ that being at a vertex $\bw$ at step $n$, 
having gotten a specified series of edge labels, following
a given lifted path  implies that:
\begin{enumerate}
\item
the steps and vertices of the lifted path have sufficient information to reconstruct
 two paths of input $y_1$ and $y_2$ producing that path; 
\item
 the first $n$ $p$-adic digit symbols of 
 $y_1+ y_2$ have been correctly computed by symbols of the steps of the path so far, namely that if 
$$
y_{i,n} = \sum_{k=0}^{n-1}  a_{k,i} p^k,~~~i=1,2,
$$
then
$$
y_{1,n}+y_{2,n} = \sum_{k=0}^{n-1} b_k p^k  + e p^{n},
$$
where $e= e_n$ is the current carry-digit, and the $b_i$ are the edge labels
produced so far in the graph $\sG_{1,2}^{+}$. 
\end{enumerate}

Suppose that  the next directed edge 
moves to a vertex $\bw'= \bw_{n+1}$, with data $(e', a')$.
Then we have $a'= \ell_1 = a_{n,1}$ and 
$$
e' = \frac{1}{p} \big( \ell_1 + \ell_2 - \ell + e\big)= \frac{1}{p}\big( a + \ell_2 -\ell +e\big)
$$
Since $e, a, \ell$ are known, this equation uniquely determines the label $\ell_2=a_{n,2}$.
Since both $\sG_1$ and $\sG_2$ are
right-resolving, the edges $(j, j')$ 
and  $(k, k')$ with the labels $a_{n_1}, a_{n,2}$ are legal steps which
uniquely determine  the  edges updating $y_{1,n}, y_{2,n}$
to $y_{1, n+1}, y_{2, n+1}$. Now the definition of edge labels in $\sG_{1,2}$
assigns the label $b_n:= \ell$ to the edge of $\sG_{1,2}$ and $e' = e_{n+1}$ in
$$
y_{1,n+1}+y_{2,n+1} = \sum_{k=0}^{n} b_k p^k  + e' p^{n},
$$
completing the induction step. 
\end{proof}

\begin{rem}
The  presentation $Y_1+Y_2 =(\bar{i}_p, \sG_{1,2}^{+}, \bw_0)$ 
 in this construction is generally far from right-resolving. 
This occurs because some values $y= y_1+y_2 \in Y_1+Y_2$
may have more than one representation $(y_1, y_2)$. This construction produces a separate path for
each pair $(y_1, y_2)$, so more than one path can yield the same sequence of labels. 
\end{rem}

%
%
%

\section{$p$-adic multiplication and  path set fractals} \label{sec5}

%
%
%

\subsection{Multiplication by $p$-integral rational numbers} \label{sec51}

We give constructive proof for multiplication by rational
numbers of specific types.

\begin{thm}\label{th51}
Let $Y$ belong to  $\sC(\ZZ_p)$ and suppose it has a
standard presentation
 $(\bar{i}_{p}, \sG, v_0)$ having $V$ vertices.  Let   $M \ge 2$
 be a positive integer with $\gcd(p, M)=1.$

 (1) For  $r= M$  the multiplicatively shifted set 
 $Y^{'} := M Y \in \sC(\ZZ_p)$. It has a 
 right-resolving presentation having at most
 $ (M+1)V$ vertices.

 (2)  For  $r= \frac{1}{M}$ 
 the multiplicatively  shifted set $Y^{'} := \frac{1}{M} Y \in \sC(\ZZ_p)$. It  has a 
 right-resolving presentation having at most
 $ (M+1)V$ vertices.

(3) For $r= -1$  the multiplicatively shifted set
$Y^{'} :=  -  Y \in \sC(\ZZ_p)$. It has a 
 right-resolving presentation having at most
 $ 2V$ vertices.
 
 (4) For $r = p^k$,  $k \ge 0$,  the  multiplicatively shifted set
$Y^{'} :=  p^k Y \in \sC(\ZZ_p)$. It has a 
 right-resolving presentation having at most $ k+ V$ vertices.
\end{thm}

\begin{proof}
We are given $Y \in \sC(\ZZ_p)$, with a standard presentation  $Y= (\bar{i}_p, \sG, v_0)$,
in which $\sG$ has $V$ vertices.
For each given $r$ we  give an explicit construction of 
  a standard 
presentation $Y'  :=(\bar{i}_p, \sG', \bw_0)$ 
and establish that $Y'= rY$ in each case.
The constructions in cases (1)-(3) are similar.

(1) Here $r=M$ with 
$p \nmid M$, and 
we construct a presentation $Y'$ of a $p$-adic path set fractal and show 
 $Y'=M Y$.
We start with an (infinite)  graph $\sG'' $ whose vertices
will be pairs $\bw= (v, e)$, in which $v$ is a vertex of $\sG$,
and  $e \ge 0$ is a carry-digit. The initial vertex is $\bw_0 :=(v_0, 0)$. 
The exit edges from a vertex $\bw$ to a vertex $\bw^{'}=(v', e')$ will occur only
if there is at least one  edge from $v$ to $v'$. Given  such an edge of $\sG$ with  label $\ell$,
we assign a corresponding  edge of $\sG'$with label $\ell'$ given  by
\begin{equation}\label{eq431}
\ell'= M\ell +e ~~(\bmod ~p), ~~~0 \le \ell' \le p-1,
\end{equation}
which is well-defined since $(p, M)=1$.
We require that the new carry digit be
\begin{equation} \label{eq432}
e'  := \frac{1}{p} \big( e +M\ell - \ell' \big).
\end{equation}

We define $(\sG', \bw_0)$ to the graph obtained taking all vertices reachable from $\bw_0$
in the above construction. We prove that all reachable vertices have carry-digit  $0\le e \le M$ by induction 
on the number of steps $n$ on a minimal path to such a vertex. 
The base case $n=0$ is true, since $e=0$, and  the induction step follows by observing
from (\ref{eq432})  that
$$
e' \le \frac{1}{p} \big( M + M(p-1)\big) \le M.
$$ 
We conclude  that the graph $\sG'$ has at most $(M+1)V$ vertices.

We now set $Y'= (\bar{i}_p, \sG', \bw_0)$, and first show this presentation is standard. We first
claim that $\sG'$ is right-resolving. We argue by contradiction.
If not, there would be two exit edges of  some vertex $\bw=(v, e)$  of $\sG$ having the same value of $\ell'$.
But then  by (\ref{eq431})
the underlying edges of $(\sG, v_0)$ would have the same value of $\ell$,
contradicting the right-resolving property of $(\sG, v_0)$. By construction $(\sG', \bw_0)$ is reachable,
hence this presentation of $Y'$ is standard.

It remains to show that  $Y^{'} = M Y$.
Consider an infinite path in $\sG$, 
with image
$$
y_{\infty} = \sum_{j=0}^{\infty} \ell_j p^j \in Y.
$$
We assert the corresponding output path 
$$
y_{\infty}^{'} = \sum_{j=1}^{\infty} \ell_j^{'} p^j \in Y'
$$
has 
$$
y_{\infty}^{'} = \frac{1}{M} y_{\infty}.
$$
We verify this by  induction on the length of a finite path approximating $y_{\infty}.$ Let
$v_0, v_1, ..., v_n$ be states on a path in $ sG$ with edge labels $\ell_0, \ell_1, \ell_{n-1}$.
Associated to this path is 
$$
y_n = \ell_0 + \ell_1 p + \cdots + \ell_{n-1} p^{n-1}.
$$
By construction we obtain
$$
y_n^{'} = \sum_{j=0}^{n-1} \ell_j^{'} p^j.
$$
Now we prove by induction on $n \ge 1$ that
$$
y_n^{'} = My_n + e_n p^n
$$
where $e_n = e'$, where $e'$ is the carry value at the final vertex $v_{n}$.
The base case $n=1$ is clear, and 
for the induction step ($e_n=e, e_{n+1}= e'$) we get, using (\ref{eq432}),
\begin{equation*} 
M y_{n+1} =  M\big(y_n + \ell_n p^n\big) = y_n^{'} +(e_n  + M\ell_n) p^n = y_{n}^{'}  + e_{n+1} p^{n+1}
\end{equation*}
Now we use $| e_n p^n|_p \to 0 $ as $n \to \infty$, whence $My_{\infty}^{'} =y_{\infty}$,
establishing the result.

(2) Here $r= \frac{1}{M}$ with   $p \nmid M$, and we construct a presentation $Y=(\bar{i}_p, \sG^{'}, \bw_0)$
of $\frac{1}{M}Y$.
We start with a (infinite) graph $\sG_M$, whose 
vertices will be pairs  $\bw= (v, e)$, in which $v$ is a vertex of $\sG$,
 and $e \geq 0$ is a carry digit, initially unbounded.
The initial vertex is $\bw_0 :=(v_0, 0)$. 
The exit edges from a vertex $\bw=(v,e)$ to a vertex $\bw=(v', e')$ will occur only
if there is at least one  edge in $\sG$ from $v$ to $v'$. Given  such an edge of $\sG$ with  label $\ell$,
we assign a corresponding  edge of $\sG'$with label $\ell'$ given  by
\begin{equation}\label{eq421}
M\ell'= \ell -e ~~(\bmod ~p), ~~~0 \le \ell' \le p-1,
\end{equation}
which is well-defined since $(p, M)=1$.
We require that the new carry digit be
\begin{equation} \label{eq422}
e'  := \frac{1}{p} \big( e +M\ell' - \ell \big).
\end{equation}

We now define $(\sG', \bw_0)$ to be the graph obtained by including only the vertices reachable from $\bw_0$
in the above construction. We now show all reachable vertices have carry-digit $0\le e \le M$, by induction 
on the number of steps $n$ on a minimal path to such a vertex. 
The base case $n=0$ holds since $e=0$, and the induction step follows by observing
from (\ref{eq422})  that
$$
e' \le \frac{1}{p} \big( M + M(p-1)\big) \le M.
$$ 
We conclude  that the graph $\sG'$ has at most $(M+1)V$ vertices.

We now define the $p$-adic path set fractal $Y':= (\bar{i}_p, \sG', \bw_0)$, and first show
this presentation is standard. 
To show $\sG'$ is right-resolving, we  argue by contradiction.
If not, there would be two exit edges of  some vertex $\bw=(v, e)$  of $\sG$ having the same value of $\ell'$.
But then  by (\ref{eq421})
the underlying edges of $(\sG, v_0)$ would have the same value of $\ell$,
contradicting the right-resolving property of $(\sG, v_0)$. It is reachable by construction,
so it is a standard presentation. 

It remains to establish that $Y' =\frac{1}{M} Y$, still
supposing $(p, M)=1$. Consider an infinite path in $\sG$, 
with image
$$
y_{\infty} = \sum_{j=0}^{\infty} \ell_j p^j \in Y.
$$
We assert the corresponding output path 
$$
y_{\infty}^{'} = \sum_{j=1}^{\infty} \ell_j^{'} p^j \in Y'
$$
has 
$$
y_{\infty}^{'} = \frac{1}{M} y_{\infty}.
$$
We verify this by induction on the length of a path
starting from the initial vertex. Let
$v_0, v_1, ..., v_n$ be vertices on a path in $\sG$ with edge labels $\ell_0, \ell_1, \ell_{n-1}$.
Associated to this path is 
$$
y_n = \ell_0 + \ell_1 p + \cdots + \ell_{n-1} p^{n-1}.
$$
By construction we obtain
$$
y_n^{'} = \sum_{j=0}^{n-1} \ell_j^{'} p^j.
$$
Now we prove by induction that
$$
M y_n^{'} = y_n + e_n p^n
$$
where $e_n = e'$ is the carry value at the final vertex $v_{n}$.
For the induction step ($e_n=e, e_{n+1}= e')$) we get, using (\ref{eq422}) 
\begin{equation*} 
My_{n+1}^{'} =  M\big(y_n^{'} + \ell_n^{'} p^n\big) = y_n +(e_n  + M\ell_n^{'}) p^n = y_{n+1}  + e_{n+1} p^{n+1}
\end{equation*}
Now we use $| e_n p^n|_p \to 0 $ as $n \to \infty$, whence $My_{\infty}^{'} =y_{\infty}$
establishing the result.

(3)  Given a standard presentation of $Y:= (\bar{i}_p, \sG, v_0)$, we construct
a standard presentation $Y' :=(\bar{i}_p, \sG'', \bw_0)$ which has
$Y':=- Y$, as follows.

The vertices of $\sG' $ will be pairs $\bw= (v, e)$, in which $v$ is a vertex of $\sG$, and
 $e$ is a carry digit, which may take values  $0$ or $-1$.
The initial vertex will be  $\bw_0 :=(v_0, 0)$. 
The exit edges from a vertex $\bw$ to a vertex $\bw^{'}=(v', e')$ will occur only
if there is at least one  edge in $\sG$ from $v$ to $v'$. Given  such an edge of $\sG$ with  label $\ell$,
we assign to it a corresponding  edge of $\sG'$ from $\bw$ to $\bw'$ with label $\ell'$ given  by
\begin{equation}\label{eq441}
\ell'= -\ell + e ~~(\bmod ~p), ~~~0 \le \ell' \le p-1,
\end{equation}
If the current vertex has $e=0$ and $\ell=0$, then the  new vertex has $\ell'=0$
and is assigned carry digit $e'=0$. If either $e=-1$ or if $e=0$ and $\ell >0$,
then the  new carry digit $e'=-1$. Once a path in  $\sG'$ reaches
a vertex with carry digit $e'= -1$, all subsequent vertices reached have carry digit $-1$.
Note that when $e=-1$ we have $-p \le -\ell -1 \le -1$ and $\ell' = p-\ell -1.$

We now let  $\sG'$  denote the part of the graph above reachable from the initial vertex $\bw_0$.
This graph has at most $2 V$ vertices. We then insert all edges between reachable vertices
produced in the construction above.

We now set $Y': =(\bar{i}_p, \sG', \bw_0)$, and as before check that this is
a standard presentation. 
We claim that $\sG'$ is right-resolving.
This is clear since the label $\ell'$ on exit edges from a vertex $\bw$ are in one-one
correspondence with labels on exit edges in $\sG$ from the associated vertex $v$,
via (\ref{eq441}). The graph $\sG'$ is reachable by construction. 

It remains to show that $Y'= -Y$.
We let
$y_n = \sum_{j=0}^{n-1} \ell_j p^j$
and
$$
y_n^{'} = \sum_{j=0}^{n-1} \ell_j p^j.
$$
We have $y_n^{'} = y_n=0$ as long as the carry digit $e=0$.
Let $\ell_r-1$ be the first nonzero digit on the path, where
the carry digit switches to $-1$. From then on 
switches to $e= -1$, we have
$$
y_n^{'} = (p- \ell_{r-1}) p^{r-1} + \sum_{j=r}^{n-1} (p-\ell_j -1)p^j = -y_n +p^n.
$$
Letting $n \to \infty$ we obtain $y_{\infty}^{'}= -y_{\infty}$, establishing
the result.

(4) Let $r=p^k$ with $n \ge 1$. For $k \ge 0$ the  set $p^k Y$ consists of modifying all symbol sequences in $Y$ 
by adding 
$k$ initial zeros.  A standard presentation $Y'= (\bar{i}_p, \sG', \bw)$ for this set
is easily obtained. Let $\sG'$ consist of $\sG$ with the addition of  
 $k$ new vertices $\bw_j \, (0 \le j \le k-1)$.
 Each of the new vertices has a single exit edge from $\bw_j$ to $\bw_{j+1}$ assigned label $0$, for $0 \le j \le k-2$,
and a similar exit edge labeled $0$ from $\bw_{k-1}$ to $v_0$. The start vertex of
$\sG'$ is $\bw_0$, and $\sG'$ has $k+V$ vertices.
\end{proof}

\begin{rem}
Theorem \ref{th51} excluded the case
 ``multiplication by $p^k$ with $k <0 $, since
 these maps do not have range in  $\ZZ_p$.

\end{rem}

%
%
%

\subsection{Proof of Theorem \ref{th16} } \label{sec52}
Theorem \ref{th16} follows immediately  from Theorem~\ref{th51}.

\begin{proof}[Proof of Theorem  \ref{th16}]
Let $r$ be a $p$-integral rational number, i.e. $ord_p(r) \ge 0$.
We may factor $r = (-1)^a p^k \frac{M_1}{M_2}$, in which $a \in \{0, 1\}, \, k \ge0$
and $\gcd (p, M_1 M_2)=1$. Now we successively apply the constructions
in Theorem \ref{th51} to multiply $Y$  by $\frac{1}{M_1}$, next multiply the
resulting set by  $M_2$, next multiply the resulting set by  $(-1)^a$, and finally 
multiply the resulting set 
by $p^k$. 
\end{proof}

%
%
%

\section{Examples }\label{sec6}

In the following examples, we let $\Sigma_p(\sD)$ denote the $p$-adic integer Cantor set consisting of
all $p$-adic integers whose digits are drawn from a given  set \newline
$D \subseteq \{0, 1, \cdots , p-1\}$.
All $\Sigma_p(\sD) \in \sC(\ZZ_p)$, and have Hausdorff dimensions \newline
\[
d_H(\Sigma_p(D)) = \log_p |D|= \frac{ \log |D|}{\log p}. 
\]


\begin{exmp} \label{ex61}
This example concerns  adding a $p$-integral rational number $r$ to  
the $3$-adic Cantor set $Y_{01}:=\Sigma_{3}(\{0,1\})$, whose $3$-adic expansions omit
the digit $2$.
It has a right-resolving presentation as
a $3$-adic path set by the pointed labeled graph $(\sG, 0)$ pictured in  Figure ~\ref{fig1}.

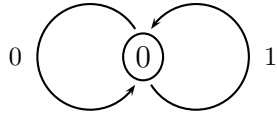
\begin{figure} [H]
	\centering
	\psset{unit=1pt}
	\begin{pspicture}(-100,0)(100,120)
		\newcommand{\noden}[2]{\node{#1}{#2}{n}}
		\noden{$0$}{0,60}
		\bcircle{n$0$}{90}{0}
		\bcircle{n$0$}{270}{1}
	\end{pspicture}
	\caption{Presentation $(\sG,0)$ of $\Sigma_{3}(\{0,1\})$.} \label{fig1}
\end{figure}

This graph has adjacency matrix 

\begin{equation*}
\bf{A} = \left(\begin{array}{cccccc}
2 \\
\end{array}\right),
\end{equation*}
whose Perron-Frobenius eigenvalue is $2$, hence the Cantor set
$\Sigma_{3}(\{0,1\})$ has Hausdorff dimension $d_H(\Sigma_{3}(\{0,1\})) = \log_3 2$, as stated above. 

Now we consider the effect of additively
shifting by  $r=2$.  The construction of Section ~\ref{sec41} applied to
the presentation above yields the
$p$-adic path set fractal
 presentation of 
$Y_{01}+2= \Sigma_{3} (\{0,1\}) + 2,$
 given in 
Figure ~\ref{fig2}, denoted 
$$Y_{01}+2 = i_3(X_{\sG'}(0200)).  $$

\begin{figure}[H]
	\centering
	\psset{unit=1pt}
\begin{pspicture}(-105,-15)(105,165)
		\newcommand{\noden}[2]{\node{#1}{#2}{n}}
		\noden{$0200$}{-60,120}
		\noden{$0000$}{60,120}
		\noden{$0011$}{-60,30}
		\noden{$0001$}{60,30}
		\bcircle{n$0000$}{0}{0}
		\bcircle{n$0001$}{180}{1}
		\dline{n$0000$}{n$0001$}{1}{0}
		\aline{n$0200$}{n$0000$}{2}
		\aline{n$0200$}{n$0011$}{0}
		\aline{n$0011$}{n$0000$}{1}
		\aline{n$0011$}{n$0001$}{2}
	\end{pspicture}
\bigskip	
	\caption{Presentation $(\sG', 0200)$ of $Y_{01}+ 2$.} \label{fig2}
\end{figure}
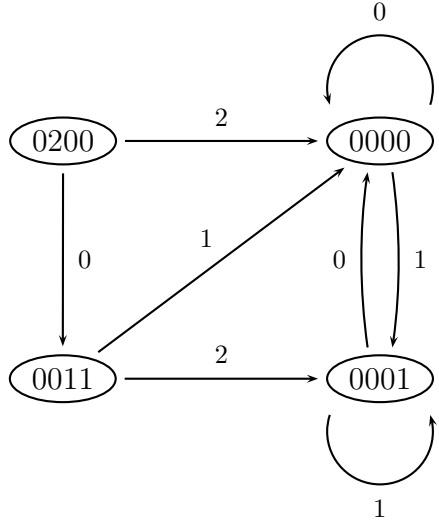

Under one ordering of the vertices of $\mathcal{G}'$, the adjacency matrix of the underlying (undirected) graph of $\mathcal{G}'$ is 
\begin{equation*}
\bf{A'} = \left(\begin{array}{cccccc}
0 & 1 & 1 & 0 \\
0 & 0 & 1 & 1 \\
0 & 0 & 1 & 1 \\
0 & 0 & 1 & 1 \\
\end{array}\right).
\end{equation*}
The eigenvalues of $\bf{A'}$ are $2$ and $0$ (multiplicity $3$), so we see the Perron  eigenvalue is $2$. 
Thus  the Hausdorff dimension is 
$$
d_H(Y_{01} + 2) = \frac{\log 3}{\log 2} = d_{H}(Y_{01})
$$
Here $Y_{01}+ \alpha$ for any $\alpha \in \ZZ_3$ must have
the same Hausdorff dimension, because they are bi-Lipschitz equivalent, hence their adjacency matrices must
have the same Perron eigenvalue. Note that  only a countable set of
values of  $\alpha$ can give  $Y_{01} + \alpha \in \sC(\ZZ_3)$, since $\sC(\ZZ_3)$ is
a countable set.
\end{exmp}

\begin{exmp}\label{ex62}
We consider the effect of set addition on $5$-adic Cantor sets $\Sigma_5(D)$ for certain subsets of
digits $D$. For all sets of two digits, we have $d_H( \Sigma_5( \{ a, b\}) = \log_5 2.$ 
Set
$$
Y_{i, j} := \Sigma_5(\{ 0,i\}) + \Sigma_5(\{ 0, j\}),  \mbox{for} ~~1 \le i, j\le 4. 
$$
One can show that
\begin{equation}\label{eq62}
\log_5 3 \le d_{H}(X_{i,j}) \le  \log_5 4.
\end{equation}
We find by inspection that the sums of certain Cantor sets are themselves Cantor sets: 
$$
Y_{1, 1}:= \Sigma_5(\{ 0,1\}) + \Sigma_5(\{ 0, 1\}) = \Sigma_5( \{ 0,1,2\})
$$
$$
Y_{1,2} =\Sigma_5(\{ 0,1\}) + \Sigma_5(\{ 0, 2\}) = \Sigma_5( \{ 0,1,2,3\})
$$

These examples  have Hausdorff dimensions 
$d_{H}(Y_{1,1}) = \log_5 3$ and
$d_{H} (Y_{1,2}) = \log_5 4$,  
respectively,  and they show that the bounds in (\ref{eq62}) are sharp.
Much more interesting are the sets $Y_{2,3}$ and $Y_{1,4}$, which are not Cantor sets;
here the  $p$-adic carry operations occuring during addition in the set sum destroy the Cantor
set property. To compute their
Hausdorff dimension, we  first find $p$-adic path set presentations for them by
the construction of Theorem \ref{th14}. These presentations are not  right-resolving, but we 
then apply the  subset construction method in \cite[Section 2]{AL12a} to obtain a right-resolving presentation.
We omit the details, noting only that  
for $Y_{1,4}$  we find the resulting graph has five vertices and adjacency matrix
\begin{equation*}
\bf{A_{14}} = \left(\begin{array}{ccccc}
1& 1 & 1 & 0 &0\\
1 & 1 & 1 & 0  &0\\
1 & 1& 0 & 1 & 1  \\
1 & 1& 0 & 1 & 1  \\
 1& 1 & 0 & 1 & 1  \\
\end{array}\right).
\end{equation*}
Its Perron eigenvalue is $2 +\sqrt 2$.
Computing its Hausdorff dimension by the formula of Theorem \ref{th31a},
we obtain 
$$d_{H}(Y_{1,4}) = \log_5 (2+ \sqrt{2})\approx \log_5 (3.41412).$$
A similar construction for $Y_{2,3}$ leads to 
$$d_H(Y_{2,3} ) =\log_5 (2+ \sqrt{3}) \approx \log_5 (3.73205).$$
\end{exmp}
\begin{exmp}\label{ex63}
We  consider on the
effect on the $3$-adic Cantor set $Y_{01} :=\Sigma_{3}(\{0,1\})$  of a multiplicative translation by $r =\frac{1}{4}.$
The set $\frac{1}{4} Y_{01}$ has a presentation 
$(\mathcal{H}, 00)$ obtained from
that of $\Sigma_{3}(\{0,1\})$ given by  $(\mathcal{G}, 0)$, 
using  the construction given in  Section ~\ref{sec51}. This presentation  is shown in Figure $~\ref{fig3}$.

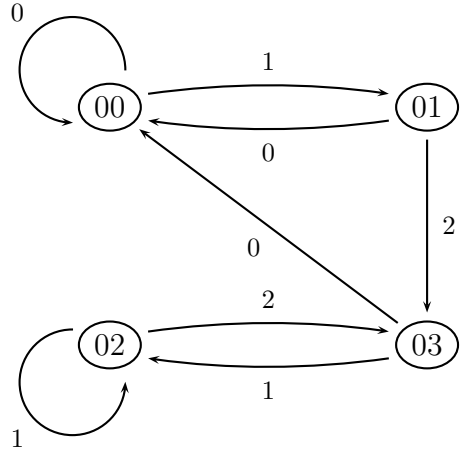
\begin{figure}[H] 
	\centering
	\psset{unit=1pt}
\begin{pspicture}(-105,-5)(105,155)
		\newcommand{\noden}[2]{\node{#1}{#2}{n}}
	\noden{$00$}{-60,120}
		\noden{$01$}{60,120}
		\noden{$02$}{-60,30}
		\noden{$03$}{60,30}
		\bcircle{n$00$}{45}{0}
		\bcircle{n$02$}{135}{1}
		\dline{n$00$}{n$01$}{1}{0}
		\aline{n$03$}{n$00$}{0}
		\aline{n$01$}{n$03$}{2}
		\dline{n$02$}{n$03$}{2}{1}
	\end{pspicture}
	\smallskip

	\caption{Presentation $(\mathcal{H},00)$ of $\frac{1}{4} \Sigma_{3} (\{0,1\})$.}\label{fig3}
\end{figure}

The adjacency matrix $\bf{B}$ of the underlying graph of $\mathcal{H}$ is
\begin{equation*}
\bf{B} = \left(\begin{array}{cccccc}
1 & 1 & 0 & 0  \\
1 & 0 & 0 & 1 \\
0 & 0 & 1 & 1  \\
1 & 0 & 1 & 0  \\
\end{array}\right).
\end{equation*}
This matrix has  Perron eigenvalue  $2$, and it has three other smaller nonzero eigenvalues,
one  real and two conjugate complex.  Using the formula in Theorem \ref{th31a}(3) we obtain
$d_H (\frac{1}{4} Y_{01}) = \log_3 2$. 
\end{exmp}
\begin{exmp}\label{ex64}
In this  example we consider the effect of intersecting multiplicatively translated 
Cantor sets taken from Example \ref{ex63}.  Let
$$
Y:= \frac{1}{4} Y_{01} \cap Y_{01} = \Sigma_3(\{0,1\})\cap \frac{1}{4} \Sigma_3(\{0,1\}).
$$ 
We obtain by the method of  \cite[Section 4]{AL12a} applied to the presentations above the
presentation of  $(\mathcal{H}',000)$ shown in Figure ~\ref{fig4}, 
where $\mathcal{H}'$ is the label product $\mathcal{H}' = \sG \star \mathcal{H}$, 
as defined in \cite[Section 4]{AL12a}.

\begin{figure}[H]
	\centering
	\psset{unit=1pt}
\begin{pspicture}(-100,-5)(100,60)
		\newcommand{\noden}[2]{\node{#1}{#2}{n}}
	\noden{$000$}{-60,30}
		\noden{$001$}{60,30}
		\bcircle{n$000$}{90}{0}
		\dline{n$000$}{n$001$}{1}{0}
	\end{pspicture}
	\caption{Presentation $(\mathcal{H}',000)$ of $\Sigma_3(\{0,1\}) \cap \frac{1}{4} \Sigma_{3}(\{ 0,1\})$.} 
	\label{fig4}
\end{figure}
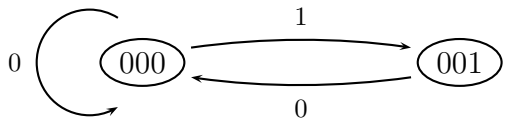

The adjacency matrix of the underlying graph of $\mathcal{H}'$ is 
\begin{equation*}
\bf{B}' = \left(\begin{array}{cccccc}
1 & 1  \\
1 & 0  \\
\end{array}\right),
\end{equation*}
whose Perron  eigenvalue is $\frac{1 + \sqrt{5}}{2}$. We conclude that 
$$
d_H(Y) = d_H(\Sigma_{3} (\{0,1\})\cap \frac{1}{4} \Sigma_{3}(\{0,1\})) = \log_3\left(\frac{1 + \sqrt{5}}{2}\right).
$$
 
 \end{exmp}

\begin{rem}
 In ~\cite{AL12c} we   will study intersections of multiplicative translates of $3$-adic Cantor sets in much more detail.

 \end{rem}
 
 
 \section{Concluding Remarks}\label{sec7}
 
 The constructions of this paper may prove interesting from the viewpoint of 
 nonnegative integer matrices and
 their eigenvalues. 
  By Theorem \ref{th12} the Hausdorff dimension
 is given by the base $p$ logarithm of the spectral radius of the underlying adjacency matrix of 
 a standard path set presentation graph, which is a nonnegative integer matrix. 
 As noted in Section \ref{sec12}, for nonzero $r \in \QQ \cap \ZZ_p$
 the maps $X \mapsto X+ r$ and $X \mapsto rX$
 preserve Hausdorff dimension. On the level of path set presentations 
 these constructions therefore produce infinitely many different
 integer matrices, of varying dimensions,
 all having  the same spectral radius, plus  various eigenvalues
 of smaller modulus
 whose cardinality and size will change under these operations.  
The allowed dimension of these matrices as the parameter $r$ varies will be unbounded.

The spectral radius of a nonnegative matrix is always attained by
a nonnegative real eigenvalue, according to the 
Perron-Frobenius theory. In the special case of nonnegative integer matrices ${\bf A}$
this maximal eigenvalue is a real algebraic integer $\beta$,
and if $\bf{A}$ is not nilpotent, then  $\beta \ge 1$.
It is termed the {\em Perron eigenvalue} in Lind and Marcus \cite[Definition 4.4.2]{LM95}. 
This eigenvalue is necessarily a weak Perron number, which is defined to
be any positive $n$-th root of some Perron number (\cite{Lin87}) for some $n \ge1$;  a
{\em Perron number} is any  real algebraic integer $\beta\ge 1$ which
is strictly larger in absolute value than all of its conjugates. 
 Lind \cite[Theorem 1]{Lin84} showed that the Perron eigenvalue of any aperiodic nonnegative  integer
matrix is a Perron number, and that conversely every Perron number occurs
as the Perron eigenvalue of some aperiodic nonnegative integer matrix. 
More generally the Perron eigenvalue of any non-nilpotent nonnegative
integer matrix is a weak Perron number, and conversely every weak Perron number
occurs as the Perron eigenvalue of at least one such  matrix. Perron numbers appear
as the topological entropies of Axiom A diffeomorphisms via a result of
Bowen \cite{Bow70}, see \cite[p. 288]{Lin84}.

 The constructions in this paper
 could be of interest  in investigating and producing examples of 
  graphs with a fixed Perron eigenvalue, particularly in case where this eigenvalue is
  very close to $1$.  In order to produce nonnegative
 matrices that have a given Perron number $\beta$ as spectral radius, it  is sometimes
 necessary to take a nonnegative matrix of dimension strictly larger than the degree of the
 minimal polynomial of $\theta$, see an example given in Lind \cite{Lin83}, \cite[Section 3]{Lin84}.
 In such cases  the characteristic polynomial of this matrix must
 contain extraneous eigenvalues. 
 The constructions of this paper offer a method to generate interesting examples
 of this kind.
Such graph constructions might also conceivably be useful in investigating conjectures on 
the smallest Perron number of each degree,
 a topic studied in  Boyd \cite{Boy85} and Wu \cite{Wu10}.
 A more speculative direction would be relating the structure of such graphs
 in connection with Lehmer's conjecture on the Mahler measure of 
 irreducible polynomials.

%
%
%

\end{document}